\documentclass[11pt,reqno,tbtags,a4paper]{amsart}
\usepackage{amssymb}
\usepackage{url}
\usepackage[square,numbers]{natbib}
\bibpunct[, ]{[}{]}{;}{n}{,}{,}

\title[Asymptotic normality in random graphs]
{Asymptotic normality in random graphs with given vertex degrees}

\date{19 December, 2018; revised 30 January, 2019}

\author{Svante Janson}
\thanks{Partly supported by the Knut and Alice Wallenberg Foundation}
\address{Department of Mathematics, Uppsala University, PO Box 480,
SE-751~06 Uppsala, Sweden}
\email{svante.janson@math.uu.se}
\urladdr{http://www.math.uu.se/svante-janson}

\subjclass[2010]{05C80, 60C05; 60F05} 

\numberwithin{equation}{section}

\renewcommand\le{\leqslant}
\renewcommand\ge{\geqslant}

\allowdisplaybreaks

\theoremstyle{plain}
\newtheorem{theorem}{Theorem}[section]
\newtheorem{lemma}[theorem]{Lemma}

\newtheorem{conjecture}[theorem]{Conjecture}
\newtheorem{claim}{Claim}

\theoremstyle{definition}
\newtheorem{example}[theorem]{Example}

\newtheorem{remark}[theorem]{Remark}

\theoremstyle{remark}

\newenvironment{romenumerate}[1][-10pt]{
\addtolength{\leftmargini}{#1}\begin{enumerate}
 }{\end{enumerate}}

\newenvironment{alphenumerate}[1][-10pt]{
\addtolength{\leftmargini}{#1}\begin{enumerate}
 }{\end{enumerate}}

\newenvironment{PXenumerate}[1]{
\addtolength{\leftmargini}{-10pt}%
\begin{enumerate}
 }{\end{enumerate}}

\newcounter{oldenumi}
{\setcounter{oldenumi}{\value{enumi}}
\begin{romenumerate} \setcounter{enumi}{\value{oldenumi}}}
{\end{romenumerate}}

\newcounter{thmenumerate}
\newenvironment{thmenumerate}
{\setcounter{thmenumerate}{0}%
 \def\item{\par
 \refstepcounter{thmenumerate}\textup{(\roman{thmenumerate})\enspace}}
}
{}

\newcounter{xenumerate}   

\newcommand\pfitemx[1]{\par#1:}
\newcommand\pfitemref[1]{\pfitemx{\ref{#1}}}

\newcounter{kasus}
\newcommand\pfcase[1]{\smallskip\refstepcounter{kasus}\noindent
 \emph{Case \arabic{kasus}: #1}}

\newcommand{\refT}[1]{Theorem~\ref{#1}}
\newcommand{\refTs}[1]{Theorems~\ref{#1}}
\newcommand{\refClaim}[1]{Claim~\ref{#1}}

\newcommand{\refL}[1]{Lemma~\ref{#1}}
\newcommand{\refLs}[1]{Lemmas~\ref{#1}}
\newcommand{\refR}[1]{Remark~\ref{#1}}

\newcommand{\refS}[1]{Section~\ref{#1}}
\newcommand{\refSs}[1]{Sections~\ref{#1}}
\newcommand{\refSS}[1]{Section~\ref{#1}}

\newcommand{\refE}[1]{Example~\ref{#1}}

\newcommand{\refConj}[1]{Conjecture~\ref{#1}}

\begingroup
  \divide\count255 by 60
  \multiply\count255 by -60
  \advance\count255 by \time
  \ifnum \count255 < 10 \xdef\klockan{\the\count1.0\the\count255}
  \else\xdef\klockan{\the\count1.\the\count255}\fi
\endgroup

\newcommand{\sumjo}{\sum_{j=0}^\infty}
\newcommand{\sumko}{\sum_{k=0}^\infty}

\newcommand{\sumk}{\sum_{k=1}^\infty}

\newcommand{\sumik}{\sum_{i=1}^k}
\newcommand{\sumin}{\sum_{i=1}^n}

\newcommand\set[1]{\ensuremath{\{#1\}}}
\newcommand\bigset[1]{\ensuremath{\bigl\{#1\bigr\}}}

\newcommand\xpar[1]{(#1)}
\newcommand\bigpar[1]{\bigl(#1\bigr)}
\newcommand\Bigpar[1]{\Bigl(#1\Bigr)}

\newcommand\lrpar[1]{\left(#1\right)}

\newcommand\xcpar[1]{\{#1\}}
\newcommand\bigcpar[1]{\bigl\{#1\bigr\}}

\newcommand\abs[1]{\lvert#1\rvert}
\newcommand\bigabs[1]{\bigl\lvert#1\bigr\rvert}
\newcommand\Bigabs[1]{\Bigl\lvert#1\Bigr\rvert}

\def\rompar(#1){\textup(#1\textup)}    
\newcommand\xfrac[2]{#1/#2}

\newcommand\parfrac[2]{\lrpar{\frac{#1}{#2}}}

\newcommand\Bigparfrac[2]{\Bigpar{\frac{#1}{#2}}}

\newcommand\innprod[1]{\langle#1\rangle}

\def\xexp(#1){e^{#1}}

\newcommand\ntoo{\ensuremath{{n\to\infty}}}

\newcommand\Ltoo{\ensuremath{{L\to\infty}}}

\newcommand\punkt{.\spacefactor=1000}    
\newcommand\iid{i.i.d\punkt}    
\newcommand\ie{i.e\punkt}
\newcommand\eg{e.g\punkt}
\newcommand\viz{viz\punkt}
\newcommand\cf{cf\punkt}
\newcommand{\as}{a.s\punkt}

\newcommand\whp{w.h.p\punkt}
\newcommand\whpx{w.h.p}

\newcommand{\tend}{\longrightarrow}
\newcommand\dto{\overset{\mathrm{d}}{\tend}}
\newcommand\pto{\overset{\mathrm{p}}{\tend}}

\newcommand\eqd{\overset{\mathrm{d}}{=}}

\newcommand\bbN{\mathbb N}

\newcounter{CC}
\newcounter{cc}

\newcommand\E{\operatorname{\mathbb E{}}}
\renewcommand\P{\operatorname{\mathbb P{}}}

\newcommand\Var{\operatorname{Var}}
\newcommand\Cov{\operatorname{Cov}}

\newcommand\Po{\operatorname{Po}}

\newcommand\ga{\alpha}
\newcommand\gb{\beta}
\newcommand\gd{\delta}
\newcommand\gD{\Delta}

\newcommand\gam{\gamma}
\newcommand\gG{\Gamma}

\newcommand\kk{\kappa}
\newcommand\gl{\lambda}
\newcommand\gL{\Lambda}
\newcommand\go{\omega}

\newcommand\gs{\sigma}
\newcommand\gS{\Sigma}
\newcommand\gss{\sigma^2}

\newcommand\gU{\Upsilon}
\newcommand\eps{\varepsilon}

\newcommand\cC{\mathcal C}
\newcommand\cD{\mathcal D}
\newcommand\cE{\mathcal E}
\newcommand\cF{\mathcal F}
\newcommand\cG{\mathcal G}

\newcommand\cJ{\mathcal J}

\newcommand\cP{\mathcal P}

\newcommand\cT{{\mathcal T}}

\newcommand\tH{\widetilde H}

\newcommand\tT{{\tilde T}}

\newcommand\tX{{\widetilde X}}

\newcommand\tZ{{\tilde Z}}

\newcommand\indic[1]{\boldsymbol1\xcpar{#1}} 
\newcommand\bigindic[1]{\boldsymbol1\bigcpar{#1}}

\newcommand\qw{^{-1}}
\newcommand\qww{^{-2}}
\newcommand\qq{^{1/2}}
\newcommand\qqw{^{-1/2}}
\newcommand\qqq{^{1/3}}

\newcommand{\pgf}{probability generating function}
\newcommand{\mgf}{moment generating function}

\newcommand\rhs{right-hand side}

\newcommand\gnp{\ensuremath{G(n,p)}}
\newcommand\gnm{\ensuremath{G(n,m)}}

\newcommand\xoo{_0^\infty}

\newenvironment{Aenumerate}[1][-10pt]{
\addtolength{\leftmargini}{#1}\begin{enumerate}

 }{\end{enumerate}}

\newenvironment{Aenumerateq}
{\setcounter{oldenumi}{\value{enumi}}
\begin{Aenumerate} \setcounter{enumi}{\value{oldenumi}}}
{\end{Aenumerate}}

\newcommand\fS{\mathfrak{S}}
\newcommand\xx[1]{^{(#1)}}
\newcommand\xxq[1]{^{#1}}
\newcommand\qpi{\hat\tau}

\newcommand\ells{i}

\newcommand\GA{A}
\newcommand\GB{B}
\newcommand\ddn{\ensuremath{\mathbf{d}_n}}
\newcommand\ddd{\ensuremath{\mathbf{d}}}

\newcommand\gnddd{\ensuremath{G(n,\ddd)}}
\newcommand\gndd{\ensuremath{G(n,\ddn)}}
\newcommand\ggndd{\ensuremath{G^*(n,\ddn)}}
\newcommand\ggnddd{\ensuremath{G^*(n,\ddd)}}
\newcommand\hgndd{\ensuremath{\widehat{G}(n,\ddn)}}

\newcommand\YY{Y^*}
\newcommand\HX{\gL}
\newcommand\HXX[1]{\HX_{#1}}
\newcommand\HXY[1]{\widehat\HX_{#1}}

\newcommand\hF{\hat F}
\newcommand\hG{\widehat G}
\newcommand\hH{\widehat H}
\newcommand\DD{\cD}

\newcommand\ZZ{Z^*}
\newcommand\bZ{\overline{Z}}
\newcommand\bZZ{\overline{Z}^*}
\newcommand\uZ{{}^{\mathsf u}Z}
\newcommand\uZZ{{}^{\mathsf u}\ZZ}
\newcommand\ubZ{{}^{\mathsf u}\bZ}
\newcommand\ubZZ{{}^{\mathsf u}\bZZ}
\renewcommand\uZ{Z^{\mathsf u}}
\renewcommand\uZZ{Z^{\mathsf u *}}
\renewcommand\ubZ{\bZ^{\mathsf u}}
\renewcommand\ubZZ{\bZ^{\mathsf u *}}
\newcommand\mue{e}
\newcommand\ffac[2]{((#1))_{#2}}
\newcommand\aut{\operatorname{aut}}
\newcommand\sC{\mathsf C}
\newcommand\sK{\mathsf K}
\newcommand\uZZsC[1]{\uZZ_{\sC_{#1}}}
\newcommand\ZZsC[1]{\ZZ_{\sC_{#1}}}
\newcommand\scx[1]{*\mathsf C_{#1}}
\newcommand\refAA{\ref{AD}--\ref{ADmu}}
\newcommand\refAAp{\ref{AD}--\ref{ADmu} and \ref{Ap1}}
\newcommand\refAAm{\ref{AD} and \ref{Am}}
\newcommand\refAAsuper{\refAA{} and \ref{Asuper}}
\newcommand\refAAsuperp{\refAA{} and \ref{Asuper}--\ref{Ap1}}
\newcommand\refAAA{\ref{AD}--\ref{AD2}}
\newcommand\refAAAp{\ref{AD}--\ref{AD2} and \ref{Ap1}}
\newcommand\refAAAsuper{\refAAA{} and \ref{Asuper}}
\newcommand\refAAAsuperp{\refAAA{} and \ref{Asuper}--\ref{Ap1}}
\newcommand\kkk{q} 
\newcommand\htH{\widehat{H}}
\newcommand\dmax{d_{\text{\rm max}}}
\newcommand\glscx[1]{\gl_{\scx{#1}}}
\newcommand\Phiu{\Phi_{u}}
\newcommand\mmg{\mmx{} multigraph}
\newcommand\mmx{marked}
\newcommand\DDX{\iota}
\newcommand\bound{bound}
\newcommand\tgss{\tilde\gs^2}
\newcommand\tgS{\tilde\gS}
\newcommand\hD{\widehat D}
\newcommand\EE{\operatorname{\E^\dagger}}
\newcommand\psiq{h}
\newcommand\Psix{\Psi^\dagger}
\newcommand\GG{G^*_n}
\newcommand\lepsn{\le\eps n}
\newcommand\sumlepsn{\sum_{\ell=L+1}^{\eps n}}
\newcommand\Di{Y_1}
\newcommand\Dii{Y_2}
\newcommand\nx{n^{\mathsf{out}}}
\newcommand\hp{\hat p}
\newcommand\PPsi{\Phi}
 \newcommand\hchi{\widehat\chi}
 \newcommand\bgS{\widehat\gS}
  \newcommand\bgs{\widehat\gs}
  \newcommand\bgss{\bgs^2}

\newcommand\glc{\mu}

\newcommand\CS{Cauchy--Schwarz}
\newcommand\CSineq{\CS{} inequality}

\newcommand\ER{Erd\H os--R\'enyi}

\begin{document}

\begin{abstract} 
  We  consider random graphs with a given degree sequence and show,
  under weak technical conditions,
  asymptotic normality of the number of components isomorphic to a given
  tree,
  first for the random multigraph given by the configuration model and then,
  by a conditioning argument, for the simple uniform random graph with the
  given degree sequence. Such conditioning is standard for convergence in
  probability, but much less straightforward for convergence in distribution
  as here. The proof uses the method of moments, and is based on a new
  estimate of mixed cumulants in a case of weakly dependent variables.

  The result on small components is applied to give a new proof of
  a recent result by  Barbour and R\"ollin on asymptotic normality of the
  size of the giant component in the random multigraph; moreover, we extend
  this to the random simple graph.
\end{abstract}

\maketitle

\section{Introduction}\label{S:intro}

Let $\gnddd$ be a random (simple) graph with $n$ labelled
vertices and a given degree
sequence $\ddd=\bigpar{d_1,\dots,d_n}$, chosen uniformly at random among all
such 
graphs. (We assume tacitly that $\ddd$ is such that some such graph
exists.) 
We will denote the vertices by $v_1,\dots,v_n$; thus $v_i$ has by definition
degree $d_i$.

The standard way to constuct a random graph $\gnddd$
is by the \emph{configuration model},
which was introduced by \citet{Bollobas-config}.
As is well-known, this method constructs first a random multigraph, which we
denote by $\ggnddd$, and then obtains $\gnddd$ by conditioning on the event
that $\ggnddd$ is simple; see \refS{Sconfig}.

We are, as most papers in this field, 
interested in asymptotic results as \ntoo, where the degree sequence
$\ddn=(d_i^{(n)})_1^n$ depends on $n$ and satisfies suitable conditions.
The standard method is to first prove results for the random multigraph
$\ggndd$ and then obtain corresponding results for $\gndd$ by conditioning
as above.
In the present paper, we make the common assumption that the (asymptotic)
degree distribution has a finite second moment; see \refS{Snot} for precise
assumptions. 
Then, it is well-known that $\P\bigpar{\ggndd\text{ is simple}}\ge c$ 
for some $c>0$ (at least for large $n$), see \refR{Rmom},
and as a consequence, any property
that holds \whp{} (i.e., with probability tending to 1) for $\ggndd$ holds
\whp{} also after conditioning on $\ggndd$ being simple.
Hence, the transfer of such results from the multigraph $\ggndd$ to the
simple graph $\gndd$ is trivial.
(Cf.\ \citet{BollobasRiordan-old}, where transfer is made possible 
by far from trivial arguments also when 
$\P\bigpar{\ggndd\text{ is simple}}\to0$.)

However, it has repeatedly been remarked that this simple method fails for
distributional results, for example that some random variable $X_n$ defined
by the graph is asymptotically normal, since probabilities like $\P(X_n\le
x)$ may be changed by the conditioning.
Hence, although in many cases it seems intuitively clear that a few loops or
multiple edges should not affect the asymptotic results, and we still expect
the same asymptotic distribution for $\gndd$ as for $\ggndd$, 
it is typically difficult to prove this rigorously.
We know only two papers where this has been done for 
asymptotic normality of some variables, in both cases
by first proving a result for $\ggndd$ and then
showing that the proof can be modified to work for $\gndd$:
\citet{SJ196} showing asymptotic normality of the size of the $k$-core 
(using a rather complicated extra argument for $\gndd$)
and 
\citet{Riordan-phase} showing
asymptotic normality of the size of the giant component in the weakly
supercritical case (using a simple extra argument for $\gndd$
noting that the proof only uses
local explorations involving $o(n)$ vertices).

The purpose of present paper  is to do this in another case.
\citet{BarbourR} recently proved, for $\ggndd$, 
a theorem on asymptotic normality of a class of ``local'' statistics that
include, for example, the number of small components of a given type.
They further used this to show asymptotic normality of the size of the giant component
in the supercritical case.
Our main results show that the same results on small components and the size
of the giant hold for the random simple graph
$\gndd$.
(We also weaken somewhat
the technical conditions for these results in \cite{BarbourR}.)
Precise statements are given in \refS{Smain} below.

We achieve these results by the time-honoured method of moments.
(\citet{BarbourR} use Stein's method.)
Using the method of moments, we show joint convergence of, e.g., the number of
components of a given type and the numbers of loops and pairs of parallel
edges in $\ggndd$; we may then obtain the result for $\gndd$ by conditioning
on the latter numbers being 0.
In order to do this, we thus show convergence of mixed moments.
Calculations of means and (co)variances are rather straightforward, and
the central part of the proof is to obtain bounds on higher-order cumulants. 
This is similar to results by \citet{Feray-Ewens,Feray}.

To describe the idea, consider  as a simple case the covariance $\Cov(I_1,I_2)$
of two indicators, each indicating that a particular set of vertices (and
half-edges) form a copy of a given graph $H$. If the two sets of vertices
are disjoint, then $I_1$ and $I_2$ are only weakly dependent; we exploit
this by constructing a 
modification $I_2'$ of $I_2$ that is independent of
$I_1$, and such that $I_2'\eqd I_2$ and $\P(I_2'\neq I_2\mid I_1)=O(\E I_2/N)$. 
This implies $\Cov(I_1,I_2)=\E(I_1I_2)-\E(I_1I_2')=\E\bigpar{I_1(I_2-I_2')}
=O\bigpar{\E I_1\E I_2/N}$.
The general case is an
extension of this, although the details are quite technical, see \refS{SX}.
(The idea to construct a suitable independent modification is used also in
the Stein coupling constructed by \cite{BarbourR}, although their 
modification is both constructed and used differently from our modifications.)

One feature of the construction that may be of independent interest is that
we in the proof use the bipartite version of the configuration model to 
construct a random bipartite graph with the vertices $[n]$ of $\ggndd$
as vertices on one side, and the edges of $\ggndd$ as vertices on the other
side; equivalently, we construct first the random bipartite graph obtained
by bisecting each edge in $\ggndd$; see \refS{Sconfig}.

\begin{remark}
  As remarked by \citet{BarbourR}, there are not many  papers at all proving
asymptotic normality for statistics of the multigraph $\ggndd$; 
apart from \cite{SJ196}, \cite{Riordan-phase} 
and \cite{BarbourR} just mentioned, we know of
\citet{AngelEtAl} (number of loops and multiple edges when the second moment
of the degree distribution is infinite; this is obviously not relevant for
simple graphs),
\citet{KhEtAl} (an epidemic on the graph),
\citet{Ball2018} (a more general epidemic model, and the giant component in
site or 
bond percolation),
and
\citet{AthreyaY}
(certain statistics in a subcritical case).
\end{remark}

\begin{remark}
  The method of moments is a very old method. Applications of it are
  typically messy and lead  to long calculations using combinatorial
  estimates of multiple sums, while other methods may give shorter and more
  elegant proofs. Nevertheless, it is a powerful method that
  often works in combinatorial problems.
I have seen several cases where results first have been proved by the method of
moments and later reproved using other methods. It seems likely that the
results here will be another example of this in the future.
For example, 
perhaps a combination of Stein's method for normal approximation and
Stein--Chen's method for Poisson approximation might be used 
instead of the method of moments to show the joint
convergence used in our proofs below.
\end{remark}

\section{Assumptions and notation}\label{Snot}

\subsection{Some notation}\label{SSnot}

$[n]:=\set{1,\dots,n}$.
$\bbN:=\set{0,1,\dots}$.
$\fS_n$ is the set of all permutations of $[n]$.

$(n)_r:=n!/(n-r)!=n(n-1)\dotsm (n-r+1)$ is the descending factorial.
Similarly, 
\begin{equation}
  \label{ffac}
\ffac{n}{r}:=
\frac{n!!}{(n-2r)!!}=
n(n-2)\dotsm(n-2(r-1))=2^r(n/2)_r.
\end{equation}

$(x)_+:=\max(x,0)$.
We interpret $0/0:=0$ and $0\cdot\infty:=0$.

Unspecified limits are as \ntoo;
\whp{} (with high probability) means with probability tending to 1 as \ntoo.
$\dto$ and $\pto$ denote convergence in distribution and probability,
respectively.
$C$ and $c$ denote positive constants that may be different at each occurrence.

If $G$ is a (multi)graph, we let $V(G)$ denote its vertex set and $E(G)$ its
edge set;
furthermore, $v(G):=|V(G)|$ and $e(G):=|E(G)|$ are the numbers of vertices
and edges of $G$, and $\kkk(G)$ is its number of components. 
The number of vertices of degree $k$ is denoted by $n_k(G)$.
For convenience, we often write $v\in G$ for $v\in V(G)$ and $|G|$ for $|V(G)|$.

The degree of a vertex $v\in G$ is denoted $d_G(v)$. (A loop in a multigraph 
contributes 2 to the degree of its endpoint.) Thus $\sum_{v\in G}d_G(v)=2e(G)$.

The component containing a vertex $v\in G$ is denoted $\cC(v)=\cC_G(v)$. It
may be regarded as a rooted multigraph, i.e., a multigraph with a
distinguished vertex, \viz{} $v$. We denote the components of $G$ arranged
in decreasing order by $\cC_1(G), \cC_2(G),\dots,\cC_{\kkk(G)}(G)$
(with ties resolved by any fixed rule, \eg{} lexicographically). 

If $X$ is a random variable and $\cE$ an event, then
$\E\xpar{X;\,\cE}:=\E\bigpar{X\indic{\cE}}$. 

We may ignore obvious roundings  and write \eg{} $n\qq$ or $\eps n$
when we mean
the nearest larger or smaller integer.

\subsection{Basic assumptions}

As in the introduction, we assume that for each $n\ge1$,
we are given a degree sequence $\ddn=(d_i)_1^n$. 
Thus $d_i=d_i\xx n$ 
depends on $n$, as do many other quantities introduced below, 
but often we omit $n$ from the notation for convenience. 
Also as above, $\gndd$ is the random simple graph with degree sequence
$\ddn$, and $\ggndd$ is the random multigraph with degree sequence $\ddn$
given by the configuration model.
(We assume tacitly that $\ddn$ is such that
a graph with these vertex degrees exists; in particular,
$\sum_i d_i$ is even.)
Let
\begin{align}\label{N}
  N:=\sumin d_i.
\end{align}
Thus the random graph $\ggndd$ has $n$ vertices and $N/2$ edges.

Let 
$n_k=n_k(\ddn):=|\set{i\in[n]:d_i=k}|$, the number of vertices of degree
$k$ in $\ggndd$.
Further, let $D_n$ denote the degree of a uniformly random vertex, \ie,
$D_n$ is a random variable with the distribution
\begin{align}\label{Dn}
  \P(D_n=k)=n_k/n, \qquad k\ge0.
\end{align}

We will always assume the following.
\begin{Aenumerate}

\item \label{AD}
$D_n$,
the degree of a randomly chosen vertex,
converges in distribution to a random variable $D$
with a finite and positive mean $\mu:=\E D$.
In other words, there exists a
probability distribution
$(p_k)_{k =  0}^\infty$
such that
\begin{equation}\label{ntopk}
  \frac{n_k}{n} \to p_k = \P(D=k),
  \qquad k \ge 0,
\end{equation}
and $\mu=\sumko kp_k\in(0,\infty)$. 

\item \label{ADmu}
$\E D_n\to \E D=\mu$. 
Assuming \ref{AD}, this is equivalent to $D_n$ being uniformly integrable.
\end{Aenumerate}

(See \eg{} \cite[Theorem 5.5.9]{Gut} 
for the equivalence with uniform integrability.)
Sometimes we will also use one or several of the following assumptions.

\begin{Aenumerateq}

\item\label{AD2} 
$\E D_n^2\to\E D^2<\infty$.
Assuming \ref{AD}, this is equivalent to $D_n^2$ being uniformly integrable.
(This will always be assumed when studying $\gndd$.)

\item \label{Am}
$\sup_n \E D_n^m<\infty$ 
for every $m<\infty$.
This implies uniform integrability of every powers $D_n^m$, and thus, assuming
\ref{AD},
$\E D_n^m\to \E D^m$ for every $m\ge0$; in particular \ref{ADmu} and \ref{AD2}.
(We use this strong condition only in a few results.)

\item \label{Asuper}
$\E D(D-2)>0$. 
This is the 
\emph{supercritical} case, when $\ggndd$ \whp{} has a giant component of order
$\Theta(n)$, see  \citet{MolloyReed95,MolloyReed98} with refinements in, \eg,
\cite{SJ204}, 
\cite{BollobasRiordan-old},
\cite{JossEtAl}.

\item \label{Ap1}
$p_1>0$ and $p_0+p_1<1$.
This excludes some less interesting cases; see \refR{RAp1}.
\end{Aenumerateq}

The degree sequences $\ddn$ will be fixed throughout the paper, and
unspecified constants may depend on them through
(explicit or implicit) constants in these assumptions, 
for example the distribution $(p_i)$.

Note that by \eqref{N}, $\E D_n=N/n$. Hence, \ref{ADmu} is equivalent to
\begin{equation}
  \label{N2}
N/n\to \mu.
\end{equation}
In particular, since $0<\mu<\infty$ by \ref{AD}, $N=\Theta(n)$. Hence, in
estimates with unspecified constants, it does not matter whether we use 
\eg{} $O(n)$
or $O(N)$.

\begin{remark}
  We assume that $\gndd$ has $n$ vertices. This is customary, but as always
  just for notational convenience. See \cite{BollobasRiordan-old}, where the 
results are formulated with the given
  $\ddn$ allowed to have arbitrary lengths $\to\infty$.
\end{remark}

\begin{remark}\label{Rrandom}
  One common version of the configuration model uses random vertex degrees
  $d_i$ that are \iid{} copies of a random variable $D$ 
(\eg{} ignoring one half-edge when the sum of degrees is odd).
Then, \ref{AD} holds a.s., and thus (with suitable assumptions on $D$),
the results below apply conditioned on the vertex degrees.
However, unconditioned results are somewhat different, see \refS{Srandom}.
\end{remark}

Let $\dmax:=\max_i d_i$. 
The uniform integrability of $D_n^2$ in \ref{AD2} means that for any
(deterministic) sequence $\go(n)\to\infty$, 
$\E D_n^2\indic{D_n\ge \go(n)}=\frac{1}{n}\sum_i d_i^2\indic{d_i\ge\go(n)}\to0$;
this implies (choosing \eg{} $\go(n)=n\qqq$) that
  \begin{align}\label{dmaxo}
    \dmax=o\bigpar{n\qq}.
  \end{align}
\ref{AD2} obviously also implies 
\begin{align}
  \label{EDn2}
\E D_n^2=O(1).
\end{align}

\begin{remark}  \label{Rmom}
 The condition \eqref{EDn2} (together with \eg{} \ref{AD}) implies
\begin{align}
  \label{pg}
\liminf_\ntoo \P\bigpar{\ggndd\text{ is simple}}>0,
\end{align}
see \cite{SJ195,SJ281}; hence \refAAA{} imply \eqref{pg}.
The lower bound \eqref{pg} is the basis for most applications of the
configuration models to the random simple graph $\gndd$, since if often
allows simple conditioning as said in the introduction.

When also \eqref{dmaxo} holds, \eqref{pg} can easily be shown using the
method of moments; this has been the standard method since the introduction
of the configuration model \cite{Bollobas-config}, see \cite{SJ195} for a
proof with no further assumptions. This method is also the basis of our proofs
below.
\end{remark}

\begin{remark}
  \label{RAD2}
We will assume \ref{AD2} for our results for the simple graph $\gndd$. 
In fact, the main results
really require only \eqref{dmaxo} and \eqref{EDn2};
by considering subsequences we may then assume that $\E D_n^2\to\mu_2$ for some
$\mu_2<\infty$, and the proofs then hold with $\E D^2$ replaced by $\mu_2$
in \eg{} \eqref{glscx}.
However, we prefer to state the results for the more natural condition
\ref{AD2}.
(An example satisfying \eqref{dmaxo} and \eqref{EDn2} but not \ref{AD2} is
obtained from any example satisfying \refAAA{} by changing the degrees of
$n\qqq$ vertices to $n\qqq$.)

As said above, it is shown in \cite{SJ195,SJ281} that \eqref{pg} holds also
without assuming \eqref{dmaxo}, but the proofs are then somewhat more
complicated, and we do not know whether they can be used to replace the
assumption \ref{AD2} by \eqref{EDn2} in the theorems below.
 In fact, it seems likely that the results extend also to cases
without \eqref{EDn2},
but that would require completely different methods.
\end{remark}

\subsection{The corresponding branching process}\label{SGW}

The usual exploration process reveals the edges (i.e., pairings of
half-edges) in $\ggndd$ one by one as they are needed, 
starting with the half-edges at a vertex $v$ and then continuing with the
unpaired half-edges at the neighbours of $v$, and so on, until the component
$\cC(v)$ is fully explored. 
Let $V$ be a uniformly random vertex in $\ggndd$.
As is well-known, see \eg{} \cite[Lemma~4]{BollobasRiordan-old} for a formal
statement,  
assuming \refAA,
the exploration process of $\cC(V)$ may be 
approximated by
a branching process $\cT$ (regarded as a rooted tree, finite or infinte), 
in the sense that for
any fixed $K$, the first $K$ generations of the two processes may be coupled
such that they are isomorphic \whpx. (In particular, if one of the processes
has only $k<K$ non-empty generations, then so has the other \whpx.)
The branching process $\cT$ is a Galton--Watson process where the root has
offspring distribution $D$ as in \ref{AD}, and all other vertices have
offspring distribution $\hD-1$ where $\hD$ has the size-biased distribution
\begin{align}
  \label{hD}
\P(\hD=k):=\hp_k:=\frac{kp_k}{\mu},
\qquad k\ge1.
\end{align}

For a rooted unlabelled tree $T$, we define
\begin{align}\label{PT}
  p_T:=\cP(\cT\cong T),
\end{align}
with $\cong$ meaning isomorphism (\ie, equality) as unlabelled rooted graphs.
Similarly, if $T$ is an unrooted unlabelled tree, we define $p_T$ by
\eqref{PT}, now interpreting $\cong$ as isomorphism (equality) of unlabelled
unrooted graphs.

For random variables that are functionals $g(\cT)$ of $\cT$, we write
\begin{align}\label{EE}
  \EE g(\cT) :=
  \E\bigpar{g(\cT);\,|\cT|<\infty}
  =\sum_Tp_Tg(T),
\end{align}
summing over finite unlabelled (rooted or unrooted) trees $T$.
Obviously, it here suffices that $g$ is defined for finite trees.

We denote the \pgf{} of $D$ by
\begin{align}
  \label{f}
f(z):=\E z^D=\sum_k p_kz^k.
\end{align}
In the supercritical case \ref{Asuper}, let $\zeta$ be the unique root
in $[0,1)$ of
\begin{align}\label{eb}
  f'(\zeta)=\mu\zeta.
\end{align}
(If \ref{Asuper} does not hold, we may take $\zeta:=1$, which always satisfies
\eqref{eb}.)
As is well-known, then $\P\bigpar{|\cT|<\infty}=f(\zeta)$.
See further \refSs{Sgiant} and \ref{Svar2}.

\section{Main results}\label{Smain}

The main results are stated below. Proofs are given in later sections.
We begin with some notation for subgraph counts.

Let $G$ and $H$ be (multi)graphs. (Think of $G$ as big and $H$ small.)
We often regard the small graph $H$ as unlabelled; in particular, when we
talk about several distinct graphs $H$, we mean non-isomorphic.
However, for formal definitions it is convenient to
regard both $G$ and $H$ as labelled, meaning that both vertices and edges, and
also half-edges, are labelled. (For simple graphs it suffices to
label vertices, since the edges are identified by their endpoints.
Similarly, half-edges need to be labelled only for loops.)
We count subgraphs of $G$ isomorphic to $H$ in four different ways:
\begin{itemize}
\item 
$Z_H(G)$ is the number of \emph{labelled} copies of $H$ in $G$.
\item 
$\uZ_H(G)$ is the number of \emph{unlabelled} copies of $H$ in $G$.
\item 
$\bZ_H(G)$ is the number of \emph{labelled isolated} copies of $H$ in $G$.
\item 
$\ubZ_H(G)$ is the number of \emph{unlabelled isolated} copies of $H$ in $G$.
\end{itemize}
Thus, $Z_H(G)$ can be defined as the
number of injective maps $H\to G$, mapping vertices to vertices, edges to
edges and half-edges to half-edges such that the relations between them are
preserved. $\bZ_H(G)$ is the number of such maps such that the vertices in
the image $H'$ of $H$ have no other edges than those in $H'$,
\ie, the degrees are preserved.
Furthermore, 
\begin{align}\label{uZ}
  \uZ_H(G)=\frac{1}{\aut(H)} Z_H(G),
&&&
  \ubZ_H(G)=\frac{1}{\aut(H)} \bZ_H(G),
\end{align}
where $\aut(H)=Z_H(H)$ is the number of automorphisms of $H$.
If $H$ is connected,
then $\ubZ_H(G)$ is the number of components of $G$ isomorphic to $H$.

The simple relations \eqref{uZ} shows that it is equivalent to study labelled
or unlabelled copies. Nevertheless we will consider both since it often is
natural to count unlabelled copies (for example when counting components),
but labelled copies often (but not always) are more convenient in our proofs.

We simplify the notation and write $Z_H:=Z_H\bigpar{\gndd}$
and
$\ZZ_H:=Z_H\bigpar{\ggndd}$, and similarly for the other versions of
subgraph counts.

\begin{example}\label{Esimple}
  Let $\sC_1$ be a loop and $\sC_2$ a double edge. Then a multigraph $G$ is
  simple if and only if $\uZ_{\sC_1}(G)=\uZ_{\sC_2}(G)=0$.
If particular, we obtain $\gndd$ from $\ggndd$ by conditioning on the event 
$\uZZ_{\sC_1}=\uZZ_{\sC_2}=0$.
\end{example}

\subsection{Small  components}

Our first theorem gives the asymptotic distributions for the number of
small components of different types in $\ggndd$.
It is well-known that most small components are trees. More precisely,
for a tree $T$, there is typically a linear number of components $T$,
and  asymptotic normality
was recently shown by \citet{BarbourR} using Stein's method.
We give a new proof of this (under somewhat weaker conditions), since this
is the basis of our work below.
We complement this with the easy results
that the number of components isomorphic to a given
conneted unicyclic graph has an asymptotic Poisson distribution, and that
there
\whp{} is no small (i.e.\ fixed size) component with more than one cycle.
The latter results follow by standard moment calculations and are
presumably known, although we do not know any specific reference.
\begin{theorem}[Mainly \cite{BarbourR}]\label{T1}
Consider the random multigraph $\ggndd$ and assume \refAAp.
  \begin{romenumerate}
  \item\label{T1t}
If\/ $H$ is a tree, then 
\begin{align}\label{t1t}
\frac{  \ubZZ_H-\E\ubZZ_H}{\sqrt n}\dto N\bigpar{0,\gss_H}
\end{align}
for some $\gss_H=\gs_{H,H}\ge0$ given by \eqref{gshh} below.
Furthermore,
\begin{equation}\label{glh}
\E\ubZZ_H/n\to
  \gl_H:=\frac{\mu^{-e(H)}}{\aut(H)}\prod_{u\in H} p_{d_H(u)} d_H(u)!
\end{equation}
and, if $v(H)>1$,
\begin{align}
  \label{t1a}
\gss_H>0 \iff
\gl_H>0\iff 
d_H(u)\in\set{k:p_k>0}\; \forall u\in H.
\end{align}
Moreover, for every tree $H$,
\begin{align}
  \label{gl=p}
\gl_H=p_H/|H|,
\end{align}
where $p_H$ is given by \eqref{PT} (for unrooted $H$).

\item \label{T1u}
If\/ $H$ is a connected unicyclic multigraph, so $e(H)=v(H)$, then
\begin{align}\label{t1u}
\ubZZ_H\dto \Po(\gl_H),
\end{align}
with $\gl_H$ as in \eqref{glh}.

\item \label{T1v}
If\/ $H$ is a connected multigraph with more than one cycle, i.e.,
$e(H)>v(H)$, then $\ubZZ_H=0$ \whp{} 
  \end{romenumerate}
Moreover, the limits in \ref{T1t}--\ref{T1v} hold jointly, for any 
finite number of distinct connected
multigraphs $H_i$, with a
joint limit $N\bigpar{0,\gS}$ for the trees, for a covariance matrix
$\gS=\bigpar{\gs_{H_i,H_j}}$ given by
\begin{align}
\gs_{H_1,H_2}
:=
\gd_{H_1,H_2}  \gl_{H_1}
+
\gl_{H_1}\gl_{H_2}
\Bigpar{\frac{2e(H_1)e(H_2)}{\mu } -\sum_{k\ge0}\frac{n_k(H_1)n_k(H_2)}{p_k}},
\label{gshh}
\end{align}
and with the Poisson limits in \ref{T1u} independent of each other
and of the joint normal limit in \ref{T1t}.
Furthermore, this holds with convergence of all mixed moments.
In particular, if $H_1$ and $H_2$ are trees, then
\begin{align}\label{cov}
  \Cov\bigpar{\ubZZ_{H_1},\ubZZ_{H_2}}=\gs_{H_1,H_2} n + o(n).
\end{align}
Moreover, if each tree $H_i$ has $v(H_i)>1$ and
satisfies the condition in \eqref{t1a}, 
then
the covariance matrix $\gS$ is non-singular.
\end{theorem}

In \eqref{gshh}, $\gd_{H_1,H_2}$ is the Kronecker delta, equal to 1 when
$H_1=H_2$ (as unlabelled graphs); furthermore, 
recall that  $0/0=0$ and $0\cdot\infty=0$.

\begin{remark}
  The case $v(H)=1$, \ie{} $v=\sK_1$, has to be excepted in \eqref{t1a},
  since trivially $\uZZ_{\sK_1}=n_0$ is deterministic, so $\gss_{\sK_1}=0$.
\end{remark}

\begin{remark}
  The moment convergence in \eqref{t1t} implies that if $H$ is a tree such
  that
$\gss_H>0$, then \eqref{t1t} can be written
  \begin{align}
    \frac{\ubZZ_H-\E\ubZZ_H}{\bigpar{\Var\ubZZ_H}\qq} \dto N(0,1),
  \end{align}
so $\ubZZ_H$ really is asymptotically normal.
The same applies in other theorems below, \eg{} \refT{T2}.
\end{remark}

\begin{remark}\label{RAp1}
  The assumption \ref{Ap1} excludes the two extremal cases $p_1=0$ and
  $p_0+p_1=1$, which are less interesting in the present context;
  in these cases, there are only a few ($o(n)$) small components
except (possibly)  isolated vertices and
edges.
\refT{T1} hold also when \ref{Ap1} fails, except that
then $\gss_{H}=0$ for every tree $H$
(and \eqref{t1a} may fail) 
and thus
\eg{} \eqref{t1t} says only that
the variable converges to 0 in probability; \cf{} \refT{T1xx} 
and \refR{Rmarked},
where \ref{Ap1} is not assumed.
\end{remark}

\begin{remark}
  Cf.\ the similar results for \ER{} graphs
  $G(n,p)$ and $G(n,m)$, shown by \eg{}
 \cite{ER}, \cite{Barbour},  \cite{BarbourKR};
  see  \refE{ER}.
\end{remark}

\begin{remark}
  In spite of \eqref{glh}, we cannot in general replace $\E\ubZZ_H$ by its
  asymptotic value $n\gl_H$ in \eqref{t1t}; the reason is that this would
  require $\E\ubZZ_H=n\gl_H+o\bigpar{n\qq}$, and this rate of convergence in
  \eqref{glh} does not hold without further assumptions on the rate of
  convergence in \eqref{ntopk}. In particular, it does not hold in the case
  of random vertex degrees mentioned in \refR{Rrandom}, see \refS{Srandom}.
\end{remark}

\begin{remark}
As said above, we reprove the result by \citet{BarbourR} on tree components.
We do not treat the more general class of local statistics studied
in \cite{BarbourR}, and leave it as an open problem whether our methods
apply to them in full generality. 
On the
other hand, we include below results on subgraph counts not covered by 
\cite{BarbourR}.
\end{remark}

One of our main results is that \refT{T1}
transfers to the simple random graph
$\gndd$, under weak assumptions.
Obviously, we have to consider only simple graphs $H$, since
$\uZ_H=0$ if $H$ contains a loop or a multiple edge.

\begin{theorem}\label{T2}
Assume \refAAAp.
Then, for simple graphs $H$,
the results in \refT{T1} hold also for the random simple graph $\gndd$
and variables $\ubZ_H$, 
  with the same $\gl_H$, $\gss_H$ and $\gS$.
Furthermore, if $H$ is a tree, then
\begin{align}\label{unorm}
  \E\ubZ_H-\E\ubZZ_H=o\bigpar{n\qq},
\end{align}
and thus it does not matter whether we normalize $\ubZ_H$ as in
\eqref{t1t} using\/  $\E\ubZ_H$ or\/ $\E\ubZZ_H$.
\end{theorem}

\subsection{Small subgraphs}\label{SSsub}

With a stronger moment assumption on the degrees, we have similar results
for the number of copies (not necessarily isolated) of a given tree.
We assume \ref{Am}, \ie, that every moment $\E D_n^m$ is bounded, and thus
in particular that $D$ has finite moments of all orders.
(This implies \ref{ADmu}--\ref{AD2} as said above.)

\begin{theorem}\label{T1xx}
Assume \refAAm.
  \begin{romenumerate}
    
  \item \label{T1x}
 For every tree $T$ there exists $\tgss_T\ge0$ such that
\begin{align}\label{t1zz}
  \frac{\uZZ_{T}-\E\uZZ_{T}}{\sqrt n}&\dto N\bigpar{0,\tgss_{T}}
.\end{align}

Moreover, \eqref{t1zz}
holds jointly for any number of
such trees $T$, with a joint limit $N(0,\tgS)$ for some covariance matrix
$\tgS$.
Furthermore, the limits hold with convergence of all moments.

\item \label{T2x}
The results in \ref{T1x} hold also for the random simple graph $\gndd$
and variables $\uZ_T$, 
  with the same 
$\tgss_T$ and $\tgS$.
Furthermore, 
\begin{align}\label{unormx}
  \E\uZ_T-\E\uZZ_T=o\bigpar{n\qq},
\end{align}
and thus it does not matter whether we normalize $\uZ_T$ as in
\eqref{t1zz} using\/  $\E\uZ_T$ or\/ $\E\uZZ_T$.
\end{romenumerate}
\end{theorem}

\begin{remark}\label{RG1}
  We consider only trees in \refT{T1xx}.
If $H$ is a cycle, then the proofs below show that $\uZZ_H$ and $\uZ_H$ are
asymptotically Poisson distributed.
However, for general unicyclic $H$ this is not true;
this is 
similar to the case
for $\gnp$, see \eg{} \cite{BollobasWierman} and \cite[Example 3.21]{JLR}.
It can be shown that for a general connected unicyclic $H$, $\uZZ_H$ and
$\uZ_H$ converge to a compund Poisson distribution.
This is besides the point of the present paper, so we leave the details to
the reader.
\end{remark}

\begin{remark}\label{RG2}
  We leave explicit calculations of asymptotics of means and variances in
\refT{T1xx} to the reader. 
Note that (for the simple graph $\gndd$)  there are further 
trivially  deterministic cases: 
$\uZ_{\sK_2}=N/2$,  the
  number of edges in $\gndd$;
similarly, for any $r\ge1$, $\uZ_{\sK_{1,r}}=\sumin \binom{d_i}{r}$. 
We do not know whether there also are further cases 
where the variance (perhaps of some linear combination) is $o(n)$ and thus
vanishes in the limit taken in \eqref{t1zz}.
\end{remark}

\begin{remark}\label{RAm}
  The condition \ref{Am} assumes that all moments of the degree distribution
are 
bounded. This is presumably stronger than necessary; it seems likely that 
it is enough that $\sup_n\E D_n^M<\infty$ for some $M$ depending on $T$ for
convergence in distribution \eqref{t1zz}, 
although \ref{Am} presumably is required for moment convergence.
Some condition of this type is necessary, since otherwise $\uZZ_T$
may be dominated by a  few vertices of high degrees, as in the following
example.
\end{remark}

\begin{example}\label{Edoublestar}
  Let $\ddn$ have one vertex $v_1$ of degree $n^{0.4}$, $n^{0.5}$ vertices
  $u_i$ of
  degree $n^{0.1}$, $n/2$ vertices of degree 3 and the rest, about $n/2$,
  of degree 1. Then \refAAA{} are satisfied, with $p_1=p_3=\xfrac12$ and
  $\mu=2$. 
Let $H$ be the tree with $v(H)=22$ obtained by taking two disjoint stars
$\sK_{1,10}$ and joining their central vertices by an edge. 

Let $X_n$ be the number of edges in $\ggndd$ between $v_1$ and a vertex of
degree $n^{0.1}$. Each such edge is the central edge in $\approx
2n^{10\cdot0.4+10\cdot0.1}=2n^5$ labelled copies of $H$, while the $O(n)$ other edges
are central edges in at most $O(n^2)$ copies of $H$ each. Hence,
$\ZZ_{H}=\bigpar{2+o(1)}n^5+O\bigpar{n^3}$. Furthermore, it is easy to see
(\eg{} by the method of moments) that $X_n\dto\Po(1/2)$.
Hence,
\begin{align}
  \ZZ_H/\bigpar{2n^5}\dto \Po\bigpar{1/2}.
\end{align}
Thus, after suitable norming, $\ZZ_H$ and $\uZZ_H$ are asymptotically
Poisson distributed and not normal. 
\end{example}

\subsection{Giant component and susceptibility}\label{Sgiant}
 
It is well-known
that in the supercritical case when \ref{Asuper} holds,
there is \whp{} a unique component $\cC_1$ of order $\Theta(n)$ in $\ggndd$ or
$\gndd$, known as the giant component;
see  \citet{MolloyReed95,MolloyReed98} and, e.g.,
\cite{SJ204}, 
\cite{BollobasRiordan-old},
\cite{JossEtAl};
moreover, $\E|\cC_1|/n\to 1-f(\zeta)$, which we recall from \refSS{SGW}
is the survival probability of the branching process $\cT$.
\citet{BallNeal} 
proved (under some extra technical conditions, in
particular the existence of a third moment $\E D^3<\infty$)
that the variance  
satisfies $\Var|\cC_1|/n\to \gss$, with $\gss$ given by \eqref{tgvar} below.
Moreover, 
\citet{BarbourR} proved  (under the same techical conditions)
that for the random multigraph $\ggndd$, the size $|\cC_1|$ is
asymptotically normal. We reprove this here with our methods (removing 
some unnecessary conditions); moreover, 
we show that  the result holds for the simple random graph $\gndd$ too.
(Similar results for \gnp{} and \gnm{} have been known for a
long time, see \eg{}
\cite{Stepanov},
\cite{Pittel}, 
\cite{PittelWormald},
\cite{BollobasRiordan-asn}.)

\begin{remark}
  In the weakly supercritical case, where $\E D_n(D_n-2)\to0$ but
 $\E D_n(D_n-2)\gg n^{-1/3}$,
\citet{Riordan-phase} showed asymptotic normality
of the size (and nullity) of the giant component, both for
$\ggndd$ and $\gndd$ (assuming that the degrees are bounded),
using  methods different from ours.
Note that in this case, the giant is smaller: $\E |\cC_1|\ll n$, but $\Var |\cC_1|
\gg n$. 
It does not seem to be possible to prove these results by the method in the
present paper.
\end{remark}

\begin{theorem}[Partly \cite{BallNeal} and \cite{BarbourR}]
  \label{TG}
  \begin{thmenumerate}    
  \item\label{TG*}
Assume \refAAsuperp. Then the size $|\cC_1|$ of the giant component in $\ggndd$
has an
asymptotically normal distribution:
\begin{align}\label{tg}
  \frac{|\cC_1|-\E |\cC_1|}{\sqrt n}\dto N(0,\gss),
\end{align}
where $\gss>0$ is given by, with $f$ and $\zeta$ as in \eqref{f}--\eqref{eb},
\begin{align}\label{tgvar}
  \gss&
  =f(\zeta)+\frac{\mu^2\zeta^2}{\mu-f''(\zeta)}
  +2 \frac{\mu^3\zeta^4}{(\mu-f''(\zeta))^2}
  -f(\zeta^2)
  -2\frac{\mu\zeta^2}{\mu-f''(\zeta)}f'(\zeta^{2})
  \notag\\&\hskip2em
  -\frac{\mu^2}{(\mu-f''(\zeta))^2}
    \bigpar{\zeta^4f''(\zeta^{2})+\zeta^2f'(\zeta^{2})}.
\end{align}
Furthermore,
$\Var|\cC_1|/n\to\gss$.

\item   \label{TG2}
Assume \refAAAsuperp.
Then the results of \ref{TG*} hold for $\gndd$ too, with the same $\gss$.
\end{thmenumerate}
\end{theorem}
The formula \eqref{tgvar} is given (in an equivalent form) by \citet{BallNeal}.
They also state that $\gss>0$, but the proof seems omitted; we give a proof
in \refS{Snull} under our (weaker) assumptions.

The proof of \refT{TG} is given in \refSs{Strunc}--\ref{Svar2};
it is based on the results above and a truncation argument; as in
\cite{BallNeal} and \cite{BarbourR}, we obtain results for the giant by
subtracting all small components, and the main ideas in this part of the
proof are similar.
As part of the proof, we show a more general result in \refT{TGpsi}; we
first introduce more notation.

A \emph{graph functional} $\psi$ is a real-valued functional $\psi(H)$
defined for unlabelled multigraphs $H$.
Let $\psi$ be a graph functional, and define for a multigraph $G$
\begin{align}
\label{Psi}
\Psi(G):=\sum_{v\in G} \psi(\cC(v))
=\sum_{j=1}^{\kkk(G)} |\cC_j|\psi(\cC_j)
=\sum_H|H|\psi(H)\ubZ_H(G),
\end{align}
summing over all unlabelled connected $H$;
the middle equality holds because each component $\cC\cong H$ is counted
$|\cC|=|H|$ times in the sum over $v$.

In the supercritical case, the second sum in \eqref{Psi} may be dominated by the
single term coming from the giant component, and it is sometimes more
interesting to exclude it and consider only small components.
We define 
\begin{align}
\label{Psix}
\Psix(G):=
\sum_{j=2}^{\kkk(G)} |\cC_j|\psi(\cC_j).
\end{align}
Recall also the notation $\EE$ defined in \eqref{EE}.

\begin{theorem}\label{TGpsi}
  \begin{thmenumerate}
    
  \item \label{TGpsi*}
Assume \refAAsuper.
  Let $\psi$ be a graph functional 
such that 
\begin{align}\label{psibound}
  \psi(H)=O\bigpar{e(H)^m+1}
\end{align}
for some constant $m$ and all connected $H$.
Define $\Psix$ by
  \eqref{Psix}.
Then
\begin{align}\label{tgpsi1}
  \E \Psix\bigpar{\ggndd}& = n\EE \psi(\cT)+o(n),
                           \\
  \Var \Psix\bigpar{\ggndd} &= n\gss_\psi+o(n),
  \label{tgpsi2}
\end{align}
and
\begin{align}\label{tgpsi}
  \frac{\Psix\bigpar{\ggndd} - \E \Psix\bigpar{\ggndd}}{\sqrt n} \dto
  N(0,\gss_\psi), 
\end{align}
where,
summing over unlabelled, unrooted trees $T_1,T_2$,
\begin{align}\label{gsspsi}
  \gss_\psi&:=\sum_{T_1,T_2}|T_1||T_2|\psi(T_1)\psi(T_2)\gs_{T_1,T_2}
\notag\\&\phantom:
=\EE\bigpar{|\cT|\psi(\cT)^2}
+\frac{2}{\mu}\bigpar{\EE \xpar{e(\cT)\psi(\cT)}}^2
-\sum_{k\ge0}\frac{1}{p_k}\bigpar{\EE\xpar{n_k(\cT)\psi(\cT)}}^2
\end{align}
with all sums and expectations absolutely convergent.
  \item \label{TGpsi0}
    Assume \refAAAsuper.
    Then the results in \ref{TGpsi*} hold for $\gndd$ too.
  \end{thmenumerate}
\end{theorem}

\begin{remark}
  If $p_k=0$ for some $k$, then $\cT$ cannot contain any vertices of degree
  $k$, so $n_k(\cT)=0$. Hence the terms in the last sum in \eqref{gsspsi}
  with $p_k=0$ vanish and are not a problem
  (recall that we interpret $0/0$  as 0);
  we could write the sum as a sum   over \set{k:p_k>0}. 
  The same applies to other sums below with $p_k$ in the denominator.
\end{remark}

\begin{remark}
  \label{RRAp1}
  We do not assume \ref{Ap1} in \refT{TGpsi}, since the results hold
  also when \ref{Ap1} fails (so $p_1=0$ by \ref{Asuper}),
  but this case is less interesting
  since then $\gss_\psi=0$, because
      $|\cT|=\infty$ \as{} and thus $\EE=0$.
\end{remark}

\begin{example}\label{Egiant}
  Let $\psi(H):=1$. Then \eqref{Psix} yields
  \begin{align}
  \Psix(G)=\sum_{j=2}^{\kkk(G)}|\cC_j|=|G|-|\cC_1|.  
  \end{align}
  Hence the asymptotic normality of $\cC_1$ in \eqref{tg}
  is equivalent to the asymptotic
  normality of $\Psix(\ggndd)$ and $\Psix(\gndd)$, which follows from
  \refT{TGpsi}, with $\gss=\gss_\psi$.
  We will show in \refL{LGvar} that this yields \eqref{tgvar},
and  in \refL{LG00} that $\gss>0$, assuming \ref{Asuper} and \ref{Ap1}.

Similarly, we may take $\psi_1(H):=e(H)/|H|$, so that \eqref{Psix} yields,
for $G=\ggndd$ or $\gndd$,
  \begin{align}
  \Psix_1(G)=\sum_{j=2}^{\kkk(G)}e(\cC_j)=N/2-e(\cC_1).
  \end{align}
\refT{TGpsi} thus shows asymptotic normality of $e(\cC_1)$; more generally,
taking a linear combinatíon $a\psi(H)+b\psi_1(H)=a+b e(H)/|H|$,
we obtain joint asymptotic normality of $|\cC_1|$ and $e(\cC_1)$.
(Cf.\ \cite{Riordan-phase} for the weakly supercritical case not studied here.)
We conjecture that the limit distribution has a non-singular covariance
matrix; however, we have not verified this; see \refR{RG0}.
\end{example}

\begin{example}\label{Esusc}
  Let $\psi(H)=|H|$; then \eqref{Psix} yields
    \begin{align}
 n\qw \Psix(G)=n\qw\sum_{j=2}^{\kkk(G)}|\cC_j|^2,
  \end{align}
which is the modified susceptibility $\hchi(G)$ studied in 
\cite{SJ241}.
It was shown there that, assuming \refAAAsuperp, $\hchi(\gndd)$ converges in
probability to  $\EE|\cT|$. \refT{TGpsi} yields convergence of
$\E\hchi(\gndd)$ to the same limit and, moreover, asymptotic normality
\begin{align}
n\qq  \bigpar{\hchi(\gndd)-\E\hchi(\gndd)}\dto N\bigpar{0,\gss},
\end{align}
for some $\gss$ given by \eqref{gsspsi}. This $\gss$ could be evaluated
explicitly 
similarly to \refL{LGvar}, but we have not done so. We show in \refE{Esusc0}
that $\gss>0$.
\end{example}

\begin{remark}
  We assume \ref{Asuper} (supercriticality) in \refTs{TG} and \ref{TGpsi}.
  This is for the truncation argument allowing us to ignore large components.
  (For example \refL{LG3}.)

  In the critical case $\E D(D-2)=0$,
  there are typically a few large components of order $\Theta(n^{2/3})$, see
  \cite[Theorem 1.3]{Riordan-phase}, and \refT{TG} does not hold; thus
  \refT{TGpsi} does not hold.

  In the subcritical case, \refT{TG} is not interesting, but other
  functionals are, for example the susceptibility in \refE{Esusc}.
  We conjecture that \refT{TGpsi} holds under suitable
  conditions, now for $\Psi$ rather than $\Psix$ (the difference should be
  negligible), but that stronger conditions than above
  (on the degree sequences or on the graph functional, or both)
are required to keep the contribution from large components small.
  Perhaps \ref{Am} will do; alternatively, \eqref{psibound} may be replaced
  by a stronger assumption.
  The paper is long as it is, and we do not consider this case further.
\end{remark}

\section{The configuration model}\label{Sconfig}
The standard way to constuct a random multigraph $\ggndd$
is by the \emph{configuration model},
which was introduced by \citet{Bollobas-config}.
(See \cite{BenderCanfield,Wormald81} for related models and arguments.)
As is well-known, we then assign a set of $d_i$ \emph{half-edges} to
each vertex $v_i$; this gives a total of $N$ half-edges, and we choose a
perfect matching of them uniformly at random. This defines $\ggndd$ by
regarding each pair of half-edges in the
matching  as an edge.

We will use this standard version in \refS{Smean}, but in the main part of our
arguments (\refS{SY}), 
it will be convenient to use a variation of this construction that
yields the same result in a somewhat more circuitous way.
To see it, we may start with the standard construction above, 
but also assume that we label the edges by putting a 
\emph{cuff} on each edge, with the cuffs labelled $1,\dots,N/2$ (uniformly
at random).
Furthermore, each cuff $i$ has two half-edges, labelled $2i-1$ and $2i$
(randomly), 
joined  to one each of the half-edges making the edge. 
(The half-edges are now really quarter-edges, but we keep the name.)

Our version of the configuration model can now be described as follows.
Let $\gU_i=\set{\go _{i1},\dots,\go_{id_i}}$ be the set of half-edges
assigned to vertex $v_i$ in the standard model above.
Let $\gU:=\bigcup_i\gU_i$ be set of all half-edges and label them
  (arbitrarily) as $\ga_1,\dots,\ga_N$.
Take also a second set of half-edges $\gb_1,\dots,\gb_N$ representing the
half-edges at the cuffs, with $\gb_{2j-1}$ and $\gb_{2j}$ at cuff $j$, which
we now for emphasis denote $\chi_j$.
Let $\pi\in\fS_N$ be a uniformly random permutation, and join each
$\ga_i$ to $\gb_{\pi(i)}$.
We denote the result by 
$\hgndd$. We interpret $\hgndd$ as a bipartite graph
by merging all half-edges in $\gU_i$ into the vertex $v_i$, and the 
half-edges  $\gb_{2j-1}$ and $\gb_{2j}$ into the cuff $\chi_j$;
thus the two vertex sets are $\set{v_i:i\in[n]}$ and $\set{\chi_j:j\in[N/2]}$.

Finally, to obtain $\ggndd$,
we  merge the two edges at each cuff into one, and forget the cuffs.
This yields evidently the same result as the standard configuration model.

The original vertices will sometimes be called \emph{real} vertices.

\begin{remark}
  The reader may recognize that our construction is just the standard 
configuration model construction of a random bipartite graph with the
real vertices on one side and the $N/2$
cuffs (each of degree 2) on the other side, followed by contractions
eliminating all cuffs.
\end{remark}

\section{Cumulants}\label{Scum}

Our proofs are based on the method of moments, in the form
using cumulants, see \eg{} \cite[Section 6.1]{JLR}.
We denote the $r$-th cumulant of a random variable $X$ by $\kk_r(X)$, and
the mixed cumulant of random variables $X_1,\dots,X_r$ by
$\kk(X_1,\dots,X_r)$.
(The variables are assumed to have finite moments.)
We recall the following properties of mixed cumulants, see \eg{}
\cite[p.~147]{JLR} or \cite{LS59}.

\begin{PXenumerate}{$\kk$}
\item\label{KK=} $
  \kk_r(X)=\kk(X,\dots,X)$ ($r$ times).
\item \label{KKmulti}
$\kk(X_1,\dots,X_r)$  is multilinear in $X_1,\dots,X_r$.
\item \label{KK0}
$\kk(X_1,\dots,X_r)=0$ if $\set{X_1,\dots,X_r}$ can be partitioned into
two non-empty sets of random variables that are independent of each other.

\item\label{KKsum}
$\kk(X_1,\dots,X_r)
 =\sum_{I_1,\dots,I_q} (-1)^{q-1}(q-1)!\prod_{p=1}^q\E\prod_{j\in I_p}X_j$,
summing over all partitions of $\set{1,\dots,r}$ into non-empty sets
$\set{I_1,\dots,I_q}$, $q\ge1$.
\item\label{KKsum2}
$\E(X_1\dotsm X_r)
 =\sum_{I_1,\dots,I_q} \prod_{p=1}^q\kk(\set{X_i:i\in I_p})$,
summing as in \ref{KKsum}.
\end{PXenumerate}

\section{The main lemma}\label{SX}
In this section we state and prove \refL{L1} below, whch is
the central part of the proof of the results in the present paper.
Although the lemma is motivated by its application to the configuration
model and random graphs, we state it in a self-contained way.
The lemma could be derived using the general theory by \cite{Feray}, and is
very similar to \cite[Theorem 1.4]{Feray-Ewens}, but we give a complete
proof, using the following notation.

Suppose that $r\ge1$ and that for each $i\in[r]$ we are given $\ell_i\ge1$ and 
two sequences
$\ga_{i1},\dots,\ga_{i \ell_i}\in[N]$ and 
$\gb_{i1},\dots,\gb_{i \ell_i}\in[N]$. 
(These are fixed throughout this section.)

Let $\pi\in\fS_{N}$ be uniformly random and define the random indicator
variables 
\begin{align}\let\ells=j
  \label{Yi}
Y_i:=\bigindic{\pi(\ga_{i\ells})=\gb_{i\ells},\forall \ells\in[\ell_i]}
=\prod_{\ells=1}^{\ell_i}\bigindic{\pi(\ga_{i\ells})=\gb_{i\ells}}.
\end{align}
Our goal in this section is to estimate the mixed cumulant
$\kk(Y_1,\dots,Y_r)$.

Let $\GA_i:=\set{\ga_{ij}: j\in[ \ell_i]}$ and
$\GB_i:=\set{\gb_{ij}: j\in [\ell_i]}$.
Let $\gG$ be the graph with vertex set $[r]$ and an edge $ik$ if 
$\GA_i\cap \GA_k\neq\emptyset$  or $\GB_i\cap \GB_k\neq\emptyset$.
(In other words, there is an edge $ik$ when $Y_i$ and $Y_k$ use a common
$\ga$ or a common $\gb$ in the definition \eqref{Yi}.)
The connected components of $\gG$ are called \emph{blocks}. Let $b$ be
the number of blocks, and denote the blocks by $\gG_1,\dots,\gG_b$ 
(in some order, \eg{}  lexicographic); 
thus $1\le b\le r$ and
$\gG_1,\dots,\gG_b$ form a partition of $[r]$.
Furthermore, let
\begin{align}\label{mue}
\mue:=\Bigabs{\bigcup_{i\in [r]}\set{(\ga_{ij},\gb_{ij}): j\in[\ell_i]}},
\end{align}
\ie, the number of distinct pairs $(\ga_{ij},\gb_{ij})$.

\begin{lemma}\label{L1}
  With notations as above,
  \begin{align}\label{l1}
\bigabs{\kk\bigpar{Y_1,\dots,Y_r)}}\le
C N^{-(b-1)-\mue}.
  \end{align}
where $C$ is a constant that may depend on $r$ and $\ell_1,\dots,\ell_r$ but not
on $N$.
\end{lemma}

\begin{remark}\label{Rif}
The estimate $\bigabs{\kk\bigpar{Y_1,\dots,Y_r)}}\le C N^{-\mue}$ is
straightforward; in particular, the case $b=1$ of \eqref{l1} is easy.
  If the indicators $\indic{\pi(\ga)=\gb}$ were independent for disjoint pairs
  $(\ga,\gb)$, then the variables $Y_i$ belonging to different blocks would
  be independent, and thus the mixed cumulant would vanish when $b\ge2$. Of
  course, in our setting these indicators are not independent, but they are
  only weakly dependent and \refL{L1} yields a substitute with an estimate
  that becomes smaller when the number of blocks gets larger.
We will see later that this is sufficient for our purposes.
\end{remark}

\begin{proof}[Proof of \refL{L1}]
In this proof, $C$ will denote various constants that may depend on
$r$ and $\ell_1,\dots,\ell_r$, but not on $N$ or $\ga_{ij},\gb_{ij}$. 
The same holds for the implied constants in $O(\dots)$.

For each set $R\subseteq[r]$, 
let $A(R):=\bigcup_{i\in R} A_i$,
$B(R):=\bigcup_{i\in R} B_i$,
and,  generalizing \eqref{mue},
\begin{align}\label{muR}
\mue(R):=\Bigabs{\bigcup_{i\in R}\set{(\ga_{ij},\gb_{ij}): j\in[\ell_i]}}.
\end{align}
Thus, $\mue:=\mue([r])$.
Note that the sets $A(\gG_1),\dots,A(\gG_b)$ are disjoint.


The idea of the proof
is, hardly surprisingly, 
that the values $\pi(\ga)$ for different $\ga$ are almost independent, 
and thus indicators $\indic{\pi(\ga)=\gb}$ for different $\ga$
are almost independent. However; this is only ``almost'', and therefore we
approximate using truly independent variables, constructed in a careful way.

We begin by constructing the random permutation $\pi$ in a special way.
Define recursively $L_1',\dots,L_b'$ and $L_1,\dots,L_b$ by
\begin{align}\label{Lk}
L_k':=1+\sum_{j=1}^{k-1} L_j, &&&
  L_k:=\bigabs{A(\gG_k)}L_k'
,&& 1\le k\le b.
\end{align}
Assume $N\ge L_b$. (Otherwise \eqref{l1} is trivial if $C$ is large
enough.)

Let $\qpi\xxq1,\dots\qpi\xxq{b}$ be independent uniformly random permutations
in $\fS_N$, and let 
$\tau\xxq k:=\qpi\xxq k(1)\dotsm\qpi\xxq k(L_k)$
be the random string consisting of the first $L_k$
values of $\qpi\xxq k$.
We write the elements of $\tau\xxq k$ as $\tau\xxq k(\ell)$, $\ell\in[L_k]$.
Using a fixed bijection between $[L_k]$ and $A(\gG_k)\times[L_k']$, we
also regard $\tau\xxq k$ as an array $\tau\xxq k(\ga,\ell)$, with  $\ga\in
A(\gG_k)$ and $\ell\in[L_k']$.

For $k\in[b]$ and  a set of indices
$J\subseteq[k-1]$,  
say that an element $\tau\xxq k(\ga,\ell)$ is \emph{bad} if it 
also occurs in $\tau\xxq j$ for some $j\in J$.
(I.e., if  it equals $\tau\xxq j(\ga',\ell')$ for some 
$(\ga',\ell')\in A(\gG_j)\times[L_j']$.) 
Finally, define 
$\pi_J\xxq k(\ga)$ for $\ga\in A(\gG_k)$ 
as $\tau\xxq k(\ga,\ell)$ for the smallest $\ell\in[L_k']$ such that this
element is not bad. Note that we have defined $L_k'$ in \eqref{Lk}
so large that there is always at least one good element for each $\ga$.
(We will not use $\pi_J\xxq k(\ga)$ for $\ga\notin A(\gG_k)$; we may define
these 
arbitrarily to make $\pi_J\xxq k$ a permutation if desired, but we may also
just ignore them.) 

 By construction, for a given $k$ and $J$, 
$\pi_J\xxq k(\ga)$ are distinct for
$\ga\in A(\gG_k)$, and these values are distinct from $\pi_I\xxq j(\ga')$ for all
$j\in J$, $I\subseteq[j-1]$ and $\ga'\in A(\gG_j)$.

In particular, consider the case
$J_k=[k-1]$ for all $k\in[b]$. 
Then all $\pi_{[k-1]}\xxq k(\ga)$, for $k\in[b]$ and $\ga\in A(\gG_k)$, are
distinct, and by symmetry, they equal any sequence of 
$\sum_k|A(\gG_k)|$
distinct values in $[N]$
with the same probability. Hence, they have the same (joint)
distribution as the values
$\pi(\ga)$ for $\ga\in \bigcup_k A(\gG_k)=A([r])$.
(Recall that the sets $A(\gG_1),\dots,A(\gG_b)$ are disjoint.)
Consequently, we may assume that 
\begin{align}
  \label{pi=}
\pi(\ga)=\pi_{[k-1]}\xxq k(\ga),
\qquad \ga\in A(\gG_k),
\, k\in[b].
\end{align} 

For $i\in[r]$, let $k(i)$ be the unique index such that $i\in\gG_{k(i)}$.
Given  $J\subseteq[k(i)-1]$,
define further a modification of $Y_i$ in \eqref{Yi} by
\begin{align}\let\ells=j
  \label{YiJ}
Y_{i;J}
:=
\prod_{\ells=1}^{\ell_i} \bigindic{\pi_{J}\xxq {k(i)}(\ga_{i\ells})=\gb_{i\ells}}.
\end{align}
Note that 
by \eqref{Yi} and \eqref{pi=},
\begin{align}\label{Y=}
  Y_i=Y_{i;[k(i)-1]}.
\end{align}

Define also a random variable denoted $Y_{i;\gD J}$
by
\begin{align}\label{gDJ}
Y_{i;\gD J} = \sum_{I\subseteq J}(-1)^{|J\setminus I|} 
Y_{i;I}.
\end{align}
This can be regarded as a kind of inclusion--exclusion, or
M\"obius inversion, and \eqref{gDJ} implies the inverse relation,
\begin{align}\label{gDJ2}
Y_{i;I} = \sum_{J\subseteq I}Y_{i;\gD J},
\qquad I\subseteq[k(i)-1].
\end{align}

We now conclude from \eqref{Y=}, \eqref{gDJ2} and the multilinearity of
mixed cumulants \ref{KKmulti} that 
\begin{align}
  \kk(Y_1,\dots,Y_r)
&
=\kk\bigpar{Y_{1;[k(1)-1]},\dots,Y_{r;[k(r)-1]}}
\notag\\&
=
\sum_{J_1\subseteq[k(1)-1]}\dotsm \sum_{J_r\subseteq[k(r)-1]}
\kk(Y_{1;\gD J_1},\dots,Y_{r;\gD J_r}).
\label{cumsum}
\end{align}
The number of terms in this sum is $O(1)$, so it suffices to estimate each
term individually.

Thus, fix a sequence $J_1,\dots,J_r$ with $J_i\subseteq[k(i)-1]$.
Let $\HX$ be the 
graph with vertex set $[b]$ and an edge $jk$ (with $j<k$)
if $j\in \bigcup_{i\in \gG_k} J_i$. 

Note first that 
a variable $Y_{i;\gD J_i}$ by \eqref{gDJ}, \eqref{YiJ} and the construction of
$\pi_J\xxq k$, depends only on $\tau\xxq {k(i)}$ and $\tau\xxq j$ for $j\in J_i$.
Consequently,
if $\HX$ is disconnected, so we may divide $\HX$ into two parts
$\HX_1$ and $\HX_2$ with no edges between them, 
then the random variables $\set{Y_{i;\gD J_i}:i\in \HX_\ell}$
depend only on $\tau\xxq k$ for $k\in \HX_\ell$,
$\ell\in\set{1,2}$,
and thus these two sets of variables are independent.
Hence, \ref{KK0} yields the following.
\begin{claim}\label{CL0}
If\/ $\HX$ is disconnected, then
$\kk\bigpar{Y_{1;\gD J_1},\dots,Y_{r;\gD J_r}}=0$.  
\end{claim}

Consequently, it suffices to consider the case when $\HX$ is connected.
For a set $R\subseteq[r]$, let
\begin{align}
\label{YYR}
  \YY_R:=\prod_{i\in R} Y_{i; \gD J_i}.
\end{align}
By \ref{KKsum}, the mixed cumulant $\kk(Y_{1;\gD J_1},\dots,Y_{r;\gD J_r})$ in
\eqref{cumsum} can be written as a linear combination of
products
\begin{align}\label{kom}
  \prod_{p=1}^q \E \YY_{R_p}
\end{align}
where $R_1,\dots,R_q$ is a partition of $[r]$. The number of terms and their
coefficients are $O(1)$, and thus it suffices to estimate each product
\eqref{kom}.

Define for $1\le j< k\le b$, $\ell\in[L_j]$ and $ m\in[L_k]$ the indicator
\begin{align}
  \label{xi}
\xi_{jk\ell m}:=\bigindic{\tau\xxq j(\ell) =\tau\xxq k(m)}
\end{align}
and let
$\Xi:=\set{\xi_{jk\ell m}}$ 
be the array of all these indicators.
Thus $\Xi$ tells us exactly which coincidences there are among the values
$\tau\xxq k(i)$. (Recall that there are no such coincidences with the same $k$.)
This also determines, for each $k$ and $J\subseteq[k-1]$, the indices of the
elements $\tau\xxq k(\ga,\ell)$ that are bad when 
constructing $\pi_J$,
and thus exactly 
which element $\tau\xxq k(\ga,\ell)$ that $\pi_J\xxq k(\ga)$ equals, for each
$\ga\in A(\gG_k)$.
This implies the following. 

\begin{claim}\label{CLXi}
  For any given sequence of sets $I_1,\dots,I_r$, $\Xi$ determines
exactly what coincidences, if any, there are among all 
$\pi\xxq {k(i)}_{I_i}(\ga)$
for $i\in[r]$ and $\ga\in A(\gG_k)$. 

Furthermore, by symmetry, 
conditioned on\/ $\Xi$,
any sequence of non-coin\-ciding values 
$\pi\xxq {k(i)}_{I_i}(\ga)$
has the same distribution as a sequence drawn without replacement from
$[N]$.
\end{claim}

Let $R\subseteq[r]$, and consider
$Z:=\prod_{i\in R}Y_{i;I_i}$. 
This is, recalling \eqref{YiJ}, a product of
factors of the type $\indic{\pi_{I_i}\xxq{k(i)}(\ga_{i\ell})=\gb_{i\ell}}$.
There may be some pairs
$(\ga_{i\ell},\gb_{i\ell})$ that are repeated;
we let $\tZ$ be the product after
deleting all factors such that $(\ga_{i\ell},\gb_{i\ell})$ repeats
a previous pair. 
Note that the number of factors remaining in $\tZ$ is $\mue(R)$ by \eqref{muR}.
Suppose that
\begin{align}\label{ga=}
\pi\xxq {k(i)}_{I_i}(\ga_{i\ell})=\pi\xxq {k(j)}_{I_j}(\ga_{jm})  
\end{align}
for two of the remaining factors
$  \indic{\pi_{I_i}\xxq{k(i)}(\ga_{i\ell})=\gb_{i\ell}}$ and 
$\indic{\pi_{I_j}\xxq{k(j)}(\ga_{jm})=\gb_{jm}}$.
If further $k(i)=k(j)$, then $\ga_{i\ell}=\ga_{jm}$, by our construction of 
$\pi\xxq{k(i)}_J$, and then $\gb_{i\ell}\neq\gb_{jm}$ since we have eliminated
repetitions. 
On the other hand, if $k(i)\neq k(j)$, then
$i$ and $j$ belong to different blocks, so by the
definition of blocks $B_i\cap B_j=\emptyset$ and thus
$\gb_{im}\neq\gb_{pq}$.
Hence, in any case, \eqref{ga=} implies $\gb_{i\ell}\neq \gb_{jm}$, and thus
\begin{align}\label{=0}
  \indic{\pi_{I_i}\xxq{k(i)}(\ga_{i\ell})=\gb_{i\ell}}\cdot
\indic{\pi_{I_i}\xxq{k(i)}(\ga_{jm})=\gb_{jm}}=0.
\end{align}

This and \refClaim{CLXi}  show
that conditioned on $\Xi$,
either $\prod_{i\in R}Y_{i;I_i}$ vanishes because the product contains conflicting
indicators, or else all 
$\pi_{I_i}\xxq{k(i)}(\ga_{ij})$ occurring in the $\mue(R)$ factors remaining
in $\tZ$
are distinct and obtained by drawing without
replacement from $[N]$.
Hence, in any case,
\begin{align}\label{eva}
  \E\Bigpar{ \prod_{i\in R}Y_{i;I_i}\Bigm|\Xi} 
\le   \E\bigpar{ \tZ\bigm|\Xi} 
\le\frac{1}{(N)_{\mue(R)}}
\le C {N^{-\mue(R)}}.
\end{align}

This is valid for any $I_i\subseteq [k(i)-1]$. 
Hence, by the definitions
\eqref{YYR} and
\eqref{gDJ} 
\begin{align}\label{ewa}
\E \bigpar{\bigabs{\YY_R}\mid\Xi}=
  \E\Bigpar{ \prod_{i\in R}\bigabs{Y_{i;\gD J_i}}\Bigm|\Xi} 
\le\sum_{I_i\subseteq J_i,\,i\in R} \E\Bigpar{\prod_{i\in R}Y_{i;I_i}\Bigm|\Xi} 
\le C N^{-\mue(R)}.
\end{align}

Suppose that $j<k$, and suppose that $\tau\xxq k$ and
$\tau\xxq j$ have no common element. 
Then, for every $J\subseteq[k-1]$,  
an element $\tau\xxq k(\ga,\ell)$ is bad for $J$ if and only if it is bad for 
$J\cup\set{j}$,
and thus
$\pi_{J\cup\set{j}}\xxq k(\ga)=\pi_{J}\xxq k(\ga)$ for all $\ga\in A(\gG_k)$;
hence  $Y_{i;J\cup\set{j}}=Y_{i;J}$ for $i\in \gG_k$.
Consequently, \eqref{gDJ} shows that then $Y_{i;\gD J}=0$ for every $J$ such that
$j\in J$.
In contrapositive form, this shows the following, recalling \eqref{xi},
\begin{claim}\label{CLmagnus}
If\/ $Y_{i;\gD J}\neq0$, $i\in\gG_k$ and $j\in J$, then 
$\tau\xxq k$ and
$\tau\xxq j$ have
at least one common element,
i.e.,
$\xi_{jk\ell m}=1$ for some $\ell\in[L_j]$ and $m\in [L_k]$.
\end{claim}

If $F$ is a  graph with vertex set $\subseteq[b]$, say that $\Xi$
\emph{covers} $F$ if for every edge $jk\in F$ with $j<k$, there exist
$\ell\in[L_j]$ and $m\in [L_k]$ such that $\xi_{jk\ell m}=1$. 
For a set $R\subseteq[r]$, let
$\HXX{R}$ be the graph with edges 
\begin{align}\label{gLR}
E\bigpar{\HXX{R}}:=
\bigcup_{k\in[b]}  \bigcup_{i\in R\cap \gG_k}\bigset{jk: j\in J_i}
\end{align}
Then, by \eqref{YYR} and \refClaim{CLmagnus}, if $\YY_R\neq0$, then
$\Xi$ covers $\HXX{R}$.
This implies that we can improve \eqref{ewa} to
\begin{align}\label{ewb}
\E \bigpar{\abs{\YY_R}\mid\Xi}
\le C N^{-\mue(R)}\indic{\text{$\Xi$ covers $\HXX{R}$}}.
\end{align}
Hence,
\begin{align}\label{ewc}
\E \bigabs{\YY_R}
\le C N^{-\mue(R)}\P\bigpar{\text{$\Xi$ covers $\HXX{R}$}}.
\end{align}

Recall that we have fixed $J_1,\dots,J_r$, so $\HXX{R}$ is for each
$R\subseteq[r]$ a fixed non-random  graph on $[b]$. 
If $F$ is any graph on $[b]$, let $\rho(F)$ be the number of edges in a
spanning forest. (This equals $b$ minus the number of components of $F$.)
We claim that if $F$ is any graph on $[b]$,
\begin{align}\label{skog}
  \P(\text{$\Xi$ covers $F$})\le C N^{-\rho(F)}.
\end{align}
To see this, we may replace $F$ by a spanning forest, so it suffices to
show \eqref{skog} when $F$ is a forest.
We do this by induction on $\rho(F)$; the base case $\rho(F)=0$ being trivial.
If $\rho(F)>0$, let $k$ be a leaf in $F$, let $k\ell$ be the
edge incident to
$k$, and let $F':=F-k\ell$. 
Recall that $\Xi$ covers $k\ell$ if $\tau\xxq k$ and $\tau\xxq\ell$ have a
common element.
Thus, conditioning on $\tau\xxq j$ for all $j\neq k$, 
the probability that $\Xi$ covers $k\ell$ is at most $L_kL_\ell N\qw \le C
N\qw$.
Hence,
\begin{align}
  \P\bigpar{\text{$\Xi$ covers $F$}}
\le C N\qw  \P\bigpar{\text{$\Xi$ covers $F'$}},
\end{align}
and \eqref{skog} follows by induction.

Combining \eqref{ewc} and \eqref{skog}, we find
\begin{align}
  \label{erika}
\E\bigabs{\YY_R} \le C N^{-\mue(R)-\rho(\HXX{R})}.
\end{align}

Finally, consider as in \eqref{kom} a partition $R_1,\dots,R_q$ of $[r]$.
Then \eqref{mue} and \eqref{muR} imply
\begin{align}\label{win}
\mue=
\mue([r])\le\sum_p\mue(R_p).  
\end{align}
Furthermore, $\bigcup_p\HXX{R_p}=\HX$ by \eqref{gLR}. 
If we take a spanning subtree
$\HXY{R_p}\subseteq\HXX{R_p}$ for each $p\le q$, then the union of these
subtrees 
has the same components as  $\bigcup_p\HXX{R_p}=\HX$.
Hence, if $\HX$ is connected, $\bigcup_p\HXY{R_p}$ 
is connected and has thus
at least $b-1$ edges,
which implies
\begin{align}\label{ston}
  \sum_{p}\rho(\HXX{R_p})\ge b-1.
\end{align}
Consequently, \eqref{erika} implies by \eqref{win}--\eqref{ston},
\begin{align}
  \label{erika2}
\Bigabs{\prod_{p=1}^q\E\YY_{R_p}} 
\le
\prod_{p=1}^q\E\bigabs{\YY_{R_p}} 
\le C N^{-\sum_p\mue(R_p)-\sum_p\rho(\HXX{R_p})}
\le C N^{-\mue-(b-1)}.
\end{align}
As said above, 
the mixed cumulant $\kk(Y_{1;\gD J_1},\dots,Y_{r;\gD J_r})$ in
\eqref{cumsum} is by \ref{KKsum} and \eqref{YYR}
a linear combination of such products, and thus
\eqref{erika2} yields
\begin{align}
  \label{eleonora}
\bigabs{\kk(Y_{1;\gD J_1},\dots,Y_{r;\gD J_r})}
\le C N^{-\mue-(b-1)}.
\end{align}
We have here assumed that $\HX$ is connected, but as said in \refClaim{CL0},
the cumulant vanishes otherwise, so \eqref{eleonora} holds in general.

Finally, also as said above, the result \eqref{l1}
follows from \eqref{cumsum} and \eqref{eleonora}.
\end{proof}

\section{Proofs of Theorems \ref{T1}, \ref{T2} and \ref{T1xx}}\label{SY}

We prove first \refT{T1} for $\ggndd$, 
and show then how the proof can be extended to \refT{T2} for
the simple graph $\gndd$.
\refT{T1xx} follows by simple modifications.

As said earlier, we use the method of moments in the form with cumulants; 
we thus show convergence of all (mixed) cumulants.
The proofs will use \refL{L1} to estimate higher cumulants, as detailed
below. In addition, we estimate means and variances (and higher moments for
unicyclic components) by standard methods; for
the sake of focussing the presentation, we state these
results as the following lemmas but postpone their proofs to \refS{Smean}.
Recall that
$\ZZ_H:=Z_H\bigpar{\ggndd}$, 

\begin{lemma}
  \label{LM}
Assume \refAA.
Then, for the random multigraph $\ggndd$:
  \begin{romenumerate}
  \item     \label{LMa}
If\/ $H$ is a tree, then
\begin{align}\label{lma}
\E \ubZZ_H = n \gl_H + o(n)
\end{align}
with $\gl_H$ given by \eqref{glh}.

  \item     \label{LMb}
If\/ $H_1,H_2$ are trees, then
\begin{align}\label{lmb}
\Cov\bigpar{\ubZZ_{H_1},\ubZZ_{H_2}}=n \gs_{H_1,H_2}+o(n),
\end{align}
where
$  \gs_{H_1,H_2}$ is given by \eqref{gshh}.

\item \label{LMc}
If\/ $H$ is a connected unicyclic multigraph, then
\begin{align}\label{lmc}
\E \ubZZ_H\to \gl_H,
\end{align}
with $\gl_H$ as in \eqref{glh}.
Moreover, for any distinct
such multigraphs $H_1,\dots,H_\ell$, and integers
$r_1,\dots,r_\ell\ge0$, the mixed factorial moments converge:
\begin{align}\label{lmc2}
  \E \prod_{i=1}^k(\ubZZ_{H_i})_{r_i} 
\to \prod_{i=1}^k \gl_{H_i}^{r_i}.
\end{align}

\item \label{LMd}
If\/ $H$ is a connected multigraph with more than one cycle, i.e.,
$e(H)>v(H)$, then 
\begin{align}
  \label{lmd}
\E\ubZZ_H\to0. 
\end{align}
  \end{romenumerate}
\end{lemma}

\begin{lemma}\label{LM1.5}
  Assume \refAA.
For every tree $H$, \eqref{gl=p} holds, \ie, 
$\gl_H=p_H/|H|$.
\end{lemma}

\begin{lemma}\label{LM2}
Assume \refAAp.
If\/ $H$ is a tree with $v(H)>1$, then $\gs_{H,H}>0\iff \gl_H>0$.
More generally,
 if\/ $H_1,\dots,H_\ell$ are distinct
trees with $v(H_i)>1$ and $\gl_{H_i}>0$ for each $i$,
then the matrix $\bigpar{\gs_{H_i,H_j}}$ is non-singular.
\end{lemma}

The next lemma shows that, assuming the second moment condition
\ref{AD2}, \refL{LM}\ref{LMc} 
extends 
to include also the numbers of loops $\uZZ_{\sC_1}$ 
and pair of multiple edges $\uZZ_{\sC_2}$, \cf{} \refE{Esimple}.
(We have to restrict to simple graphs $H_i$; if $H$ is not simple, there is
obviously a strong dependence between $\ubZZ_H$ and $\uZZ_{\sC_1}$ or
$\uZZ_{\sC_2}$.) 

\begin{lemma}
  \label{LM3}
Assume \refAAA.
Let $H_1,\dots,H_k$ be connected unicyclic simple graphs.
Then, for any integers $s_1,s_2,r_1,\dots,r_k\ge0$, 
\begin{align}\label{lm3}
  \E \Bigpar{
\bigpar{\uZZ_{\sC_1}}_{s_1}\bigpar{\uZZ_{\sC_2}}_{s_2}
\prod_{i=1}^k(\ubZZ_{H_i})_{r_i} 
}
\to \glscx1^{s_1}\glscx2^{s_2}
\prod_{i=1}^k \gl_{H_i}^{r_i},
\end{align}
where
\begin{align}
\glscx1:=\frac{\E D(D-1)}{2},
&&&
\glscx2:=\Bigparfrac{\E D(D-1)}{2}^2.
\label{glscx} 
\end{align}
\end{lemma}

As said above, we postpone the proofs of these lemmas to \refS{Smean}.

Fix a (multi)graph $H$.
Let $h:=v(H)\ge1$, and assume that the vertices of
$H$ are labelled $1,\dots,h$. 
Our aim is to estimate the cumulants $\kk_r(\ZZ_H)$ for (fixed) $r\ge1$
and corresponding mixed cumulants.
In this section, $C$ denotes constants that may depend on the order $r$ and
the fixed (small) graph $H$ (and corresponding graphs below), 
but not on $n$;
the same holds for the implicit constants in $O(\dots)$.

We use the version of the configuration model described in \refS{Sconfig}.
A copy of $H$ in $\ggndd$ then corresponds to a copy of $\hH$ in $\hgndd$,
where $\hH$ is obtained from $H$ by subdividing each edge into two, and
regarding the new vertices as cuffs, and where we only count copies of $\hH$
that map real vertices to real vertices and cuffs to cuffs.
We consider also the identity of the half-edges used in the construction,
and see that an isolated labelled
copy of $\hH$ in $\hgndd$ is described by the following
data: 
\begin{PXenumerate}{$\phi$}
\item \label{phi1}
For each vertex $i\in V(H)=[h]$: a real vertex $v_{\nu(i)}$ such that
$d_{\nu(i)}= d_H(i)$.
Furthermore, $v(1),\dots,v(h)$ are distinct.
\item \label{phi2}
For each edge $ij$ in $H$: two half-edges $\ga_{ij}\in \gU_{\nu(i)}$ and
$\ga_{ij}'\in \gU_{\nu(j)}$,
and also
a cuff $\chi_{ij}$ and a labelling of the two half-edges at $\chi_{ij}$ as
$\gb_{ij}$ and $\gb'_{ij}$.
Furthermore, all these half-edges for $ij\in E(H)$ are distinct.
\end{PXenumerate}
Each such family of data 
$\phi:=\bigpar{\nu(i),\ga_{ij},\ga_{ij}',\chi_{ij},\gb_{ij},\gb_{ij}'}$ defines
a possible isolated labelled copy $\hH_\phi$ of $\hH$.
Thus, if $I_\phi$ is the indicator that $\hH_\phi$ exists in $\hgndd$,
and $\Phi(\hH)$ is the set of all such data $\phi$,
then 
\begin{align}\label{zaa}
\bZZ_H=
\bZ_H(\ggndd)=\sum_{\phi\in\Phi(\hH)}I_\phi.
\end{align}
Furthermore, $\hH_\phi$ exists in $\ggndd$ if and
only if the construction of $\hgndd$ yields edges 
$\ga_{ij}\gb_{ij}$ and $\ga_{ij}'\gb_{ij}'$ for each $ij\in E(H)$,
which is equivalent to 
$\pi(\ga_{ij})=\gb_{ij}$ and $\pi(\ga_{ij}')=\gb_{ij}'$.
Consequently, 
each $I_\phi$ is the product of $2e(H)$ indicators 
of the type $\indic{\pi(\ga)=\gb}$.
Each $I_\phi$ is thus a random variable of the type in \eqref{Yi}.

Consider now a sequence $H_1,\dots,H_r$ of multigraphs.
By \eqref{zaa} and multilinearity \ref{KKmulti}, the mixed cumulant 
$\kk\bigpar{\bZZ_{H_1},\dots,\bZZ_{H_r}}$ 
can be expanded
as
\begin{align}\label{zac}
\kk\bigpar{\bZZ_{H_1},\dots,\bZZ_{H_r}}
=\sum_{\phi_1\in\Phi(\hH_1)}\dotsm \sum_{\phi_r\in\Phi(\hH_r)}
 \kk\bigpar{I_{\phi_1},\dots,I_{\phi_r}},
\end{align}
where the mixed cumulants may be estimated by \refL{L1}.

It remains  to estimate
the parameters $b$ and $\mue$ in \refL{L1} and the corresponding number of terms
in \eqref{zac}. This is done in the following lemmas, using a standard
type of argument that is common in applications of the method of moments in
combinatorial problems.

\begin{lemma}
  \label{LL1}
Assume \refAA.
Let $H_1,\dots,H_r$ be a sequence of connected multigraphs.
Then
\begin{align}\label{zag}
\bigabs{\kk\bigpar{\bZZ_{H_1},\dots,\bZZ_{H_r}}}
\le C\sum_{\cF} n^{v(\cF)-e(\cF)-\kkk(\cF)+1},  
\end{align}
where we sum over all unlabelled bipartite multigraphs $\cF$ that can be
written as a 
union $\bigcup_{i=1}^r \hH_i'$ where $\hH_i'\cong \hH_i$ and, as above,
$\hH_i$ is obtained from $H_i$ by subdividing each edge into two.
\end{lemma}
We consider only unions $\bigcup_{i=1}^r \hH_i'$ respecting the bipartition
between real vertices and cuffs. Note that the set of $\cF$ in \eqref{zag}
is finite and independent of $n$.

\begin{proof}
For notational simplicity we consider the case of
a single multigraph $H$,
\ie,
$H_1=\dots=H_r=H$;
the proof for mixed cumulants is the same.

Consider one term in the sum in \eqref{zac}; it is given by
indicators $I_{\phi_1},\dots,I_{\phi_r}$ corresponding  to
$r$ copies $\hH_{\phi_1},\dots,\hH_{\phi_r}$ of $\hH$.
Let $\hF:=\bigcup_{i=1}^r\hH_{\phi_i}$.
We apply \refL{L1} with $Y_i:=I_{\phi_i}$.
Note that $\mue$ in \eqref{mue} equals $e(\hF)$.
If the graph $\gG$ in \refS{SX} has an edge $jk$, then $\hH_{\phi_k}$ and
$\hH_{\phi_j}$ have a common half-edge, and thus a common vertex (real or cuff);
hence $\hH_{\phi_k}$ and $\hH_{\phi_j}$ are subgraphs of the same component
of $\hF$.
This yields a surjective map from the blocks of $\gG$ to the components of
$\hF$, and thus the number of blocks $b$ in \refL{L1} satisfies $b\ge
\kkk(\hF)$.
Hence, \refL{L1} yields
\begin{align}\label{zad}
  \bigabs{\kk_r\bigpar{I_{\phi_1},\dots,I_{\phi_r}}}
\le C N^{1-\kkk(\hF)-e(\hF)}
\le C n^{1-\kkk(\hF)-e(\hF)}.
\end{align}

Now let us count the number of $\phi_1,\dots,\phi_r\in\Phi(\hH)$ that yield a
union $\hF$ isomorphic to some given bipartite multigraph $\cF$.
Each cuff in $\hF$ and its 2 half-edges may be chosen in $O(N)=O(n)$ ways.
Let $w$ be a real vertex in $\cF$, and let $m:=d_{\cF}(w)$ be its degree.
Also, let $K:=\max_{i\in H}d_H(i)<\infty$.
We consider by definition only copies of $\hH$ such that each vertex $v$
has degree equal to the corresponding vertex in $H$, and thus $d_{\hG}(v)\le
K$. Hence, 
for each such choice of $v$ corresponding to $w\in \cF$, we have at most
$K^m$ choices of the $m$ half-edges incident to it; hence the number of
choices of $v$ and its $m$ half-edges is $O(n K^m)=O(n)$.

Consequently,  each vertex (real or cuff) in $\cF$ gives
$O(n)$ choices of corresponding vertex and half-edges in  $\hF$.
Hence, the number of $\hF$ isomorphic to a given $\cF$ is $O\bigpar{n^{v(\cF)}}$.
Finally, given $\hF$, we can choose $\hH_1,\dots,\hH_r$ in $O(1)$ ways.
Hence, the total number of terms in \eqref{zac} corresponding to a given $\cF$
is $O\bigpar{n^{v(\cF)}}$, and by \eqref{zad},
their total contribution is $O\bigpar{n^{v(\cF)-e(\cF)-\kkk(\cF)+1}}$.
This yields \eqref{zag}.
\end{proof}

\begin{lemma}\label{LZ1}
  Let $F:=\bigcup_{i=1}^r H_i$, where
$H_1,\dots,H_r$ is a sequence of connected multigraphs.
\begin{romenumerate}
\item \label{LZ1a}
Then $v(F)\le e(F)+\kkk(F)$.
\item \label{LZ1b}
If furthermore at least one $H_i$ is not a tree, then
$v(F)\le e(F)+\kkk(F)-1$.
\end{romenumerate}
\end{lemma}

\begin{proof}
Consider first the case when $F$ is connected.
We may then reorder $H_1,\dots,H_r$ such that 
each $\bigcup_{i=1}^k H_i$ is connected;
furthermore, we may do this starting with any of the multigraphs as $H_1$,
and in \ref{LZ1b} we may thus assume that $H_1$ is not a tree.

Now, choose for each $i\ge2$ a spanning tree $T_k\subseteq H_k$, and 
define 
$F_k:=H_1\cup\bigcup_{i=2}^k T_i$, $k\ge1$. Thus each $F_k$ is connected.

Let $k>1$ and let $T_k':=T_k\cap F_{k-1}$. Since $T_k'$ is a subgraph of
$T_k$, which is a tree, we see that $T_k'$ is a forest, and 
since $F_k=T_k\cup F_{k-1}$ is connected, 
$T_k'$ is not empty. Hence, $v(T_k')\ge e(T_k')+1$.
Consequently, since $F_k=T_k\cup F_{k-1}$,
\begin{align}\label{ly1}
  v(F_k)-e(F_k)
&=v(F_{k-1})+v(T_k)-v( T_k')
-\bigpar{e(F_{k-1})+e(T_k)-e(T_k')}
\notag\\&
\le v(F_{k-1})-e(F_{k-1})+v(T_k)-e( T_k)-1
\notag\\&
= v(F_{k-1})-e(F_{k-1})
\end{align}
when $k>1$. 
Hence, by induction $v(F_k)-e(F_k)\le v(H_1)-e(H_1)$.
Furthermore, $F_r$ is a spanning subgraph of $F$, and thus
$v(F_r)=v(F)$ and $e(F_r)\le e(F)$. Consequently,
\begin{align}\label{ly2}
v(F)-e(F)\le v(F_r)-e(F_r)\le v(H_1)-e(H_1).  
\end{align}
Moreover, $e(H_1)\ge v(H_1)-1$ since $H_1$ is connected, and if $H_1$ is not
a tree, then $e(H_1)\ge v(H_1)$. 
Consequently, \eqref{ly2} yields \ref{LZ1a} and \ref{LZ1b} in the case
$\kkk(F)=1$.

If $\kkk(F)>1$, \ie, $F$ is disconnected, denote the components of $F$ by
$F_i$, $i=1,\dots,\kkk(F)$.
Then, by what just has been shown, $v(F_i)\le e(F_i)+1$, and if some
$H_j\subseteq F_i$ is not a tree, then $v(F_i)\le e(F_i)$.
Hence, the result fullows by summing over all components $F_i$.
\end{proof}

\begin{lemma}
  \label{LL2}
Assume \refAA.
Let $H_1,\dots,H_r$ be a sequence of connected multigraphs.
Then
\begin{align}\label{zag1}
\kk\bigpar{\bZZ_{H_1},\dots,\bZZ_{H_r}}
=O(n).
\end{align}
Furthermore, if at least one $H_i$ is not a tree, then
\begin{align}\label{zag0}
\kk\bigpar{\bZZ_{H_1},\dots,\bZZ_{H_r}}
=O(1).
\end{align}
\end{lemma}

\begin{proof}
  An immediate consequence of \refLs{LL1} and \ref{LZ1}, applying the latter
  to $\hH_1,\dots,\hH_r$.
\end{proof}

\begin{proof}[Proof of \refT{T1}]
Define 
\begin{align}\label{XH}
  X_H:=
  \begin{cases}
    \bigpar{\ubZZ_H-\E\ubZZ_H}/n\qq,& \text{$H$ is a tree},\\
X_H:=\ubZZ_H,& \text{$H$ has a cycle}. 
  \end{cases}
\end{align}

First, \ref{T1v}, the case when $e(H)>v(H)$, is easy.
By \refL{LM}\ref{LMd} and Markov's inequality, 
$\P(X_H\neq 0)\le \E X_H\to0$.
(Recall that $X_H$ is a non-negative integer.)
In particular, $X_H\pto0$.
Furthermore, \eqref{zag0} shows that every cumulant $\kk_r(X_H)=O(1)$, and
thus every moment is bounded by \ref{KKsum2}.
This implies uniform integrability of every power, and thus 
$\E X_H^r\to0$ for every $r$. Convergence (to 0) of joint moments with other
$X_H$  follows by the \CSineq{} when we have shown convergence of moments
also in \ref{T1t} and \ref{T1u}.

For \ref{T1t} and \ref{T1u}, and joint convergence of both, we use the
method of moments, in the cumulant version.
Let $\tX_H$ be random variables defined for (unlabelled) trees
and unicyclic multigraphs $H$ to have the claimed joint limit distribution:
\begin{romenumerate}
\item For trees $T$,
 $\tX_T$ have a joint
normal distribution with $\E X_T=0$ and
$\Cov\bigpar{\tX_{T_1},\tX_{T_2}}=\gs_{T_1,T_2}$.
\item \label{lpu}
For unicyclic $F$, 
$\tX_F\sim \Po(\gl_F)$ with $\tX_F$ independent of all other $\tX_H$.
\end{romenumerate}
We then claim that for any $r\ge1$ and trees or connected unicyclic
multigraphs $H_1,\dots, H_r$,
\begin{align}\label{kklim}
  \kk\bigpar{X_{H_1},\dots,X_{H_r}} \to   \kk\bigpar{\tX_{H_1},\dots,\tX_{H_r}}.
\end{align}
Indeed, this implies by \ref{KKsum2} convergence of all mixed moments.
Furthermore, normal and Poisson distributions have finite \mgf{s}, 
and thus the joint distribution of any finite number of $\tX_H$ is
determined by  its mixed moments. Hence, \eqref{kklim} implies convergence
\begin{align}\label{dlim}
\bigpar{X_{H_1},\dots,X_{H_r}} \dto   \bigpar{\tX_{H_1},\dots,\tX_{H_r}}
\end{align}
as claimed in the theorem. Furthermore, all moments converge, as just seen.

Note that the \rhs{} in \eqref{kklim} vanishes except in the two cases
$r\ge1$ and all $H_i$ are the same unicyclic multigraph, or
$r=2$ and both $H_1$ and $H_2$ are trees. This follows by the independence
assumption in \ref{lpu} above together with \ref{KK0}, and the
fact that all mixed cumulants of order $r\ge3$ vanish for joint normal
distributions. 

In order to show \eqref{kklim}, we consider three cases.

\pfcase{Every $H_i$ is a tree.}\label{pfc-tree}
Consider three subcases.
First, if $r=1$, then 
\begin{align}
  \kk\bigpar{X_{H_1}} = \E X_{H_1} =0 =\kk\bigpar{\tX_{H_1}}.
\end{align}

Secondly, if $r=2$, then \eqref{XH} and \eqref{lmb} yield
\begin{align}
  \kk\bigpar{X_{H_1},X_{H_2}} 
&=\Cov\bigpar{X_{H_1},X_{H_2}} 
=n\qw \Cov\bigpar{\ubZZ_{H_1},\ubZZ_{H_2}} 
\notag\\&
\to \gs_{H_1,H_2}
=\kk\bigpar{\tX_{H_1},\tX_{H_2}} .
\end{align}
Thirdly, if $r\ge3$, then \eqref{zag1} in \refL{LL2} together with
\eqref{uZ} and \eqref{XH} yields
\begin{align}\label{ibam}
   \kk\bigpar{X_{H_1},\dots,X_{H_r}} 
= C n^{-r/2}   \kk\bigpar{\bZZ_{H_1},\dots,\bZZ_{H_r}} 
= O\bigpar{n^{1-r/2}} \to0.
\end{align}
This verifies \eqref{kklim} in each of the three subcases.

\pfcase{Every $H_i$ is unicyclic.}\label{pfc-uni}
This case is really nothing new. 
By Lemma \ref{LM}\ref{LMc}, we have convergence
of all mixed factorial moments to the corresponding moments of
$\tX_{H_1},\dots,\tX_{H_r}$, which by well-known 
algebraic identities is equivalent to convergence of all mixed moments, 
and thus to convergence of all mixed cumulants.
(See \cite[Section 6.1]{JLR} and \ref{KKsum}--\ref{KKsum2}.)

\pfcase{At least one $H_i$ is a tree and at least one is unicyclic.}
Suppose that there are $\ell\ge1$ trees (not necessarily distinct) among
$H_1,\dots,H_r$. Then,
by \eqref{zag0} in \refL{LL2} together with
\eqref{uZ} and \eqref{XH},
\begin{align}\label{waz}
   \kk\bigpar{X_{H_1},\dots,X_{H_r}} 
= C n^{-\ell/2}   \kk\bigpar{\bZZ_{H_1},\dots,\bZZ_{H_r}} 
= O\bigpar{n^{-\ell/2}} \to0.
\end{align}
Hence, \eqref{kklim} holds in this case too.

This completes the verification of \eqref{kklim} in all cases; as said
above, this proves \eqref{dlim}, with convergence of all moments.

Finally, \refL{LM} shows \eqref{glh},
\refL{LM1.5} shows \eqref{gl=p},
\refL{LM2} shows the first equivalence in \eqref{t1a}, and the
second is obvious from \eqref{glh}. \refL{LM2} yields also the final claim
on non-singularity of the covariance matrix $\gS=\bigpar{\gs_{H_i,H_j}}$.
\end{proof}

To show \refT{T2} for the simple random graph $\gndd$, 
we combine \refT{T1} with asymptotics of the counts 
$\uZZsC1$ and $\uZZsC2$
of loops and double edges.
Note that $\uZZ_{\sC_j}$ counts all occurences of $\sC_j$, which in general
differs from $\ubZZ_{\sC_j}$,  the number of isolated occurences. The latter
is already included in \refT{T1}, but $\uZZ_{\sC_j}$ requires an extra
argument.

$\ZZsC1$ and $\ZZsC2$ can be expressed as sums \eqref{zaa} for sets
$\Phi$ of data $\phi$ as above, with the difference that in \ref{phi1}
above, we omit the condition on the degree $d_{\nu(i)}$. 

We may just as well consider a somewhat more general situation:
we say that a {\emph \mmg} 
$\tH=\bigpar{H,(\DDX_w)_{w\in H}}$  
is a multigraph $H$ where each vertex $v$ has a mark 
$\DDX_v\in\set{\text{\emph{bound}, \emph{free}}}$.
We then define $Z_{\tH}(G)$ as the number of labelled copies of $H$ in $G$
such that each \bound{}  vertex $w\in H$ corresponds to a vertex $v$ in $G$ with
the same degree $d_G(v)=d_H(w)$ (while there is no restriction for a free
vertex). Hence, if all vertices in $H$ are free, $Z_{\tH}=Z_H$,
and if all vertices are bound, $Z_{\tH}=\bZ_H$.
We write, as in the unmarked case, $Z_{\tH}:=Z_{\tH}(\gndd)$ and
$\ZZ_{\tH}:=Z_{\tH}(\ggndd)$.

Assume that $V(H)=[h]$ for
convenience.
Then $\ZZ_{\tH}$ is given by \eqref{zaa} for a set $\Phi(\tH)$ of $\phi$
defined by
\ref{phi1}--\ref{phi2}, but with the degree condition in \ref{phi1} omitted
for free vertices.
Let cuffs in subdivided multigraphs be \bound{} by default.
Furthermore,  define $\bigcup_i\tH_i$, where $\tH_i$ are \mmg{s}, 
by taking the union and marking a vertex
as \bound{}
if it is \bound{} in some $\tH_i$, and 
free otherwise.

We extend \refL{LL1} as follows.

\begin{lemma}
  \label{LL11}
Assume \refAA.
Let ${\tH}_1,\dots,{\tH}_r$ be a sequence of connected \mmg{s}.
Then
\begin{align}\label{zagfree}
\bigabs{\kk\bigpar{\ZZ_{{\tH}_1},\dots,\ZZ_{{\tH}_r}}}
\le C\sum_{\cF} n^{v(\cF)-e(\cF)-\kkk(\cF)+1}
\prod_{w\in\cF \text{ \rm is free}} \E D_n^{d_{\cF}(w)}
\end{align}
where we sum over all unlabelled bipartite multigraphs $\cF$ that can be
written as a 
union $\bigcup_{i=1}^r \hH_i'$ where $\hH_i'\cong \htH_i$ and
$\htH_i$ is obtained from $\tH_i$ by subdividing each edge into two
as above.
\end{lemma}

\begin{proof}
We follow the proof of \refL{LL1}.
The only difference is when counting the number of ways we can choose a copy $v$
of a vertex $w\in \cF$ and its $m:=d_{\cF}(w)$ half-edges. 
If $w$ is \bound, then as in the proof of \refL{LL1},
$d_{\hG}(v)\le K$ for some $K<\infty$, 
and then the $m$ half-edges can be chosen in at most
$K^m=O(1)$ ways for each $v$, giving as before $O(n)$ choices of vertex and
half-edges. 

On the other hand, if $w$ is free, then we  bound for each real vertex 
$v=v_i\in V(\hG)$ 
the number of choices of the $m$ half-edges by $d_{\hG}(v_i)^m=d_i^m$;
hence, the
total number of choices of $v$ and its half-edges is at most
\begin{align}\label{coco}
\sum_{i=1}^n d_i^m = n \E D_n^m   = n \E D_n^{d_{\cF}(w)}. 
\end{align}
This gives an additional factor $\E D_n^{d_{\cF}(w)}$ for each free $w$, and
thus \eqref{zagfree} instead of \eqref{zag}.
\end{proof}

\begin{lemma}\label{LZ2}
  Let $F=F'\cup F''$ with
$F'=\bigcup_{i=1}^r H_i$, where $r\ge0$ and
$H_1,\dots,H_r$ is a sequence of connected multigraphs,
and $F''=\bigcup_{i=1}^t C_i$, where $t\ge1$ and
each $C_i$ is a cycle
$\sC_\ell$, $\ell\ge1$.
Say that a vertex $w\in V(F'')\setminus V(F')$ is free, and let
\begin{align}\label{sF}
  s(F):=\frac12\sum_{w \text{ \rm free}}\bigpar{d_F(w)-2}_+.
\end{align}
Then
$v(F)+s(F)\le e(F)+\kkk(F)-1$.
\end{lemma}

\begin{proof}
Let $S(F):=v(F)+s(F)-e(F)-\kkk(F)$; thus the result is 
$S(F)\le-1$. 

We use induction on $t$. 
If $t=1$, then each free vertex $w$ has degree $d_F(w)=d_{C_1}(w)=2$, 
and thus $s(F)=0$.
Hence, the result follows from \refL{LZ1}\ref{LZ1b} applied to
$\set{H_i}\cup\set{C_1}$.

Now suppose that the results hold for some $t$, 
and add another cycle
$C_{t+1}$, of length $\ell$, say. Denote the old $F$ by $F_t$, so $F=F_t\cup
C_{t+1}$.
We consider the changes in the quantities $v(F), e(F),\kkk(F),s(F), S(F)$, which
we denote by $\gD v, \gD e$, and so on. We treat three cases separately.

(i). If the new cycle $C_{t+1}$ is disjoint from $F_t$, then
$\gD v=\ell$, $\gD e=\ell$, $\gD\kkk=1$ and $\gD s=0$; the latter since all
new vertices have degree 2 and thus do not contribute to $s$.
Hence, $\gD S=-1$.

(ii). Suppose that $C_{t+1}$ is edge-disjoint but not vertex-disjoint from
 $F_t$. Then $\gD e = \ell$ and $\gD\kkk= 0$.
Each vertex in $C_{t+1}\setminus F_t$ contributes 1 to $\gD v$ and 0 to
$\gD s$, 
while each vertex in $C_{t+1}\cap F_t$ contributes 0 to $\gD v$ and at most 
1 to $\gD s$, since the new cycle increases the degree by 2. (If the vertex
is not free in $F_t$, the contribution to $\gD s$ is 0.)
Hence, each vertex contributes at most 1 to $\gD(v+s)$, and thus
$\gD(v+s)\le \ell=\gD e$. Consequently, $\gD S\le 0$.

(iii). Suppose that $C_{t+1}$ is has some edge in common with $F_t$. 
Then $\gD\kkk=0$. The edges in $E(C_{t+1})\setminus E(F_t)$ form $k\ge0$
disjoint paths $P_j$. Suppose that $P_j$ contains $\ell_j\ge1$ edges. Then $P_j$
contributes $\ell_j$ to $\gD e$, the $\ell_j-1$ internal vertices in $P_j$
contribute $1$ each to $\gD v$ and the two endpoints of $P_j$ contribute 
at most $1/2$ each to $\gD s$. There are no other contributions, and thus
$\gD (v+s)\le\sum_j\ell_j=\gD e$ and consequently $\gD S\le 0$.

We have shown that $\gD S\le0$ in all cases, which completes the induction.
\end{proof}

\begin{lemma}\label{LLJ}
  Assume \refAAA.
Let $H_1,\dots,H_r$ be trees or connected unicyclic simple graphs.
Then, the joint limits in distribution in \refT{T1}\ref{T1t}\ref{T1u}
hold jointly with
\begin{align}
  \uZZ_{\sC_1}&\dto\Po\bigpar{\glscx1}, \label{llj1} 
\\
 \uZZ_{\sC_2}&\dto\Po\bigpar{\glscx2},
\label{llj2} 
\end{align}
with $\glscx1$ and $\glscx2$ given by \eqref{glscx},
and with the limits in \eqref{llj1} and \eqref{llj2} independent of each
other and of the limits for $\ubZZ_{H_i}$ in \refT{T1}\ref{T1t}\ref{T1u}.
\end{lemma}

\begin{proof}
We extend the proof of \refT{T1}, and show that the convergence \eqref{kklim} of
joint cumulants holds also if we consider besides the variables $X_{H_i}$ also
$X_{\scx1}:=\uZZ_{\sC_1}$ and $X_{\scx2}:=\uZZ_{\sC_2}$ (possibly repeated several
times), and the corresponding $\tX_{\scx1}\sim\Po\xpar{\glscx1}$ and
$\tX_{\scx2}\sim\Po\xpar{\glscx2}$, independent of each other and all
$\tX_H$. We regard $\scx{j}$, $j=1,2$, 
as symbols used only to denote these  variables; they are
not any graphs; for
convenience we may say $H=\scx{j}$, but it should be interpreted in this 
formal sense.

We have in the proof of \refT{T1} proved \eqref{kklim} 
when there is no $\scx j$ among $H_1,\dots,H_r$, by considering three
different cases separately.
We now consider two further cases.

\pfcase{Some $H_i$ is a $\scx j$, and some $H_j$ is a tree.}
Let $\ell\ge1$ be the number of $H_i$ that are trees.
Regard each $H_i$ as a \mmg{} $\tH_i$; if  $H_i=\scx j$ 
we let every vertex be free, and otherwise we let every vertex be \bound.
By \eqref{uZ}, applied also to
$\uZZsC{j}$, we find similarly to \eqref{waz}
\begin{align}\label{mti}
   \kk\bigpar{X_{H_1},\dots,X_{H_r}} 
= C n^{-\ell/2}   \kk\bigpar{\ZZ_{\tH_1},\dots,\ZZ_{\tH_r}}, 
\end{align}
which we estimate by \eqref{zagfree}.
If $d\ge2$, then \ref{AD2} implies, through \eqref{dmaxo} and \eqref{EDn2},
\begin{align}\label{emu}
  \E D_n^d\le \dmax^{d-2}\E D_n^2= O\bigpar{n^{(d-2)/2}}.
\end{align}
Furthermore, if $\cF=\bigcup \hH_i'$ is as in \refL{LL11}, then a vertex
$w\in\cF$ is free if and 
only if it belongs to $\hH_i'$ only for  $H_i=\scx j$.
In particular, if $w\in\cF$ is free then $d_{\cF}(w)\ge2$.
Let $s(\cF)$ be as in \eqref{sF} with $F=\cF$, $F''$ the union of
$\hH_i'$ for $H_i=\scx j$ and $F'$ the union of the other $\hH_i$.
Then
\refL{LL11}, \eqref{emu}, \eqref{sF} and \refL{LZ2} yield
\begin{align}\label{zagga}
\bigabs{\kk\bigpar{\ZZ_{{\tH}_1},\dots,\ZZ_{{\tH}_r}}}
\le C\sum_{\cF} n^{v(\cF)-e(\cF)-\kkk(\cF)+1+s(\cF)}
\le C.
\end{align}
Consequently, \eqref{mti} yields
\begin{align}\label{wati}
   \kk\bigpar{X_{H_1},\dots,X_{H_r}} 
= O\bigpar{ n^{-\ell/2}}
\to0.
\end{align}
Hence, \eqref{kklim} holds in this case too.

\pfcase{Some $H_i$ is a $\scx j$, but no $H_j$ is a tree.}
This is similar to Case \ref{pfc-uni}.
\refL{LM3} shows convergence of all mixed factorial moments, which is
equivalent to convergence of all mixed moments and of all mixed cumulants.

This shows that \eqref{kklim} holds in all cases, which implies joint
convergence in distribution \eqref{dlim} as above.
\end{proof}

\begin{proof}[Proof of \refT{T2}]
\refT{T1}\ref{T1v} (multicyclic $H$) transfers immediately by \eqref{pg}.
Consider thus only trees and unicyclic $H$.

The joint convergence in distribution in \ref{T1t} and \ref{T1u} for $\gndd$ is
  an immediate consequence of \refL{LLJ} and conditioning on 
$\uZZ_{\sC_1}=\uZZ_{\sC_2}=0$. (We keep the normalization by $\E\ubZZ_H$ in
\eqref{t1t}.)

Furthermore, with $X_H$ as in \eqref{XH}, we have by \refT{T1} convergence
of every moment $\E X_H^r$, and thus $\E X_H^r=O(1)$.
Furthermore, \eqref{pg} holds  
(as said in \refR{Rmom}, and a consequence of \eqref{llj1}--\eqref{llj2}); 
hence for every even
integer $r$,
\begin{align}
  \E\bigpar{X_H^r\mid \uZZ_{\sC_1}=\uZZ_{\sC_2}=0}
\le \frac{\E X_H^r}{\P\bigpar{\uZZ_{\sC_1}=\uZZ_{\sC_2}=0}}
=O(1).
\end{align}
In other words, the conditioned random variables 
$\bigpar{X_H\mid \ggndd\text{ is simple}}$ have bounded moments of arbitrary
  order; 
hence the convergence in distribution of these variables just shown implies 
that moment convergence holds also for the conditioned variables, i.e., for
$\gndd$. Convergence of mixed moments follows by the same argument. 

In particular, this shows that for a tree $H$,
\begin{align}
\E\bigpar{X_H\mid \ggndd\text{ is simple}}=
  \frac{\E\ubZ_H-\E\ubZZ_H}{n\qq}\to0,  
\end{align}
which shows \eqref{unorm}, and completes the proof.
\end{proof}

\begin{proof}[Proof of \refT{T1xx}]
  \pfitemref{T1x}
  As the proof of \refT{T1} (only Case \ref{pfc-tree} is relevant), but using
  \refL{LL11} (with all vertices free) instead of \refL{LL2}; the extra
  factors in \eqref{zagfree} are all $O(1)$ by the assumption \ref{Am},
so this does not affect the rest of the proof.

  \pfitemref{T2x}
  As the proof of \refT{T2}, using again \refL{LL11} and \ref{Am} instead of
  \refL{LL2}. 
\end{proof}

\begin{remark}\label{Rmarked}
\refT{T1xx} extends, with the same proof, 
to \mmx{} trees $\tT$, where as
  above each vertex is marked as either \bound{} or free.
This includes \refTs{T1}, \ref{T2} and \ref{T1xx}, but also
mixed cases.
More generally, we may consider \mmg{s} where each vertex $w$ is marked with
a set
$\DD_w\subseteq\bbN$ of allowed degrees; we count only copies such that
each $w$ corresponds to a vertex $v\in G$ with $d_G(v)\in\DD_w$.
For example, we may count edges such that one endpoint has prime degree and
the other is a leaf.
The proofs above hold for this case too; we say that a vertex $w$ is bound
if $\DD_w$ is finite, and free otherwise.
If some vertex is free we assume \ref{Am} as in \refT{T1xx};
if all vertices are \bound, the assumptions \refAA{} or \refAAA{} in
\refTs{T1} and \ref{T2} are enough.

%
\end{remark}

\section{Means and variances}\label{Smean}

In this section we prove 
\refLs{LM}--\ref{LM3} used in the proofs in \refS{SY}.
We will here use the standard version of the configuration model, recalled
at the beginning of \refS{Sconfig}; we thus use the half-edges
$\gU=\bigcup_i\gU_i=\set{\ga_j}_1^N$ (see \refS{Sconfig})
and take a random perfect matching of them.

\begin{proof}[Proof of \refL{LM}]
Let $H$ be a 
multigraph and let $h_k:=n_k(H)$ be the number of vertices of
degree $k$ in $H$, $k\ge0$.
We argue as in \refS{SY}; we denote the possible isolated labelled
copies of $H$ in
$\ggndd$ by $\set{H_\phi}_{\phi\in\Phi(H)}$,
and obtain in analogy with \eqref{zaa} 
\begin{align}\label{moa}
  \bZZ_H=\sum_{\phi\in\Phi(H)} I_\phi,
\end{align}
where now $I_\phi$ is the indicator that $H_\phi$ exists as an isolated
subgraph in $\ggndd$.  $H_\phi$ is specified by:
\begin{PXenumerate}{$\phi'$}
\item \label{phix1}
For each vertex $i\in V(H)=[h]$: a vertex $v_{\nu(i)}$ such that
$d_{\nu(i)}= d_H(i)$.
Furthermore, $v(1),\dots,v(h)$ are distinct. (As \ref{phi1}.)
\item \label{phix2}
For each vertex $i\in V(H)=[h]$: also a bijection between the $d_H(i)$
half-edges at $i\in H$ and the half-edges $\gU_{\nu(i)}$ at $v_{\nu(i)}$.
These bijections define how the half-edges in $\bigcup_{i\in [h]}\gU_{\nu(i)}$ 
are paired in $H_\phi$.
\end{PXenumerate}
For each $k\ge0$, 
there are $h_k$ vertices $i$ such that $d_H(i)=k$, and we 
may choose the corresponding $\nu(i)$ in \ref{phix1} in $(n_k)_{h_k}$ ways;
then for each of these vertices
the bijection of half-edges in \ref{phix2} may be chosen in
$k!$ ways.
Hence, the number of $H_\phi$ is 
\begin{align}\label{mea}
  \abs{\Phi(H)}=\prod_{k\ge0} (n_k)_{h_k}k!^{h_k}.
\end{align}
(The product in \eqref{mea} is really finite, since $h_k\neq0$ only for finitely
many $k$.)
Each $I_\phi$ is a product of $e(H)$ indicators of specific pairings of 
the type 
$\indic{\ga_k \text{ and }\ga_\ell \text{ are paired in }\ggndd}$.
Hence,
\begin{align}\label{mia}
  \E I_\phi =\frac{1}{(N-1)(N-3)\dotsm(n-2e(H)+1)}=\frac{1}{((N-1))_{e(H)}}.
\end{align}
Consequently, \eqref{moa}, \eqref{mea} and \eqref{mia} yield the exact
formula
\begin{align}\label{bko}
  \E \bZZ_H = 
\frac{1}{((N-1))_{e(H)}}\prod_{k\ge0} (n_k)_{h_k}k!^{h_k}.
\end{align}

As \ntoo, we obtain from \eqref{bko} using \ref{AD} and 
$\gl_H$ defined in \eqref{glh},
\begin{align}
  \E \bZZ_H
&=N^{-e(H)}\bigpar{1+O(N\qw)}\prod_{k\ge0}\bigpar{np_k+o(n)}^{h_k}k!^{h_k} 
\notag\\&
=
\frac{n^{v(H)}}{N^{e(H)}}\Bigpar{\prod_{k\ge0}\bigpar{p_k k!}^{h_k} +o(1)}
\notag\\&=
{n^{v(H)-e(H)}}{\mu^{-e(H)}}\Bigpar{\prod_{u\in H}p_{d_H(u)} d_H(u)! +o(1)}
\notag\\&=
{n^{v(H)-e(H)}}\aut(H)\bigpar{\gl_H +o(1)}.
\label{mo}
\end{align}
Using \eqref{uZ}, this proves
\eqref{lma}, \eqref{lmc} and \eqref{lmd}.

Now consider two connected multigraphs $H$ and $H'$, 
and let $h_k':=n_k(H')$.
By \eqref{moa},
\begin{align}\label{mob}
  \bZZ_H\bZZ_{H'}=\sum_{\phi\in\Phi(H)}\sum_{\phi'\in\Phi(H')} I_\phi I_{\phi'}.
\end{align}
Let $\phi\in\Phi(H)$ and condition on
$I_{\phi}=1$. Then $H_\phi$ is a component of $\ggndd$, and the rest of the
graph, \ie{} $\ggndd\setminus H_\phi$, is given by another instance of the
configuration model, with $n_k$ replaced by $n_k-h_k$ and $N$ by $N-2e(H)$.
Consequently, if we first only consider $\phi\in \Phi(H)$ and
$\phi'\in\Phi(H')$ such that $H_\phi$ and $H_{\phi'}$ are disjoint, then,
using \eqref{bko} for $\ggndd\setminus H_\phi$,
\begin{align}
\hskip4em&\hskip-4em
\E  \sum_{\phi,\phi':H_\phi\cap H_{\phi'}=\emptyset}I_\phi I_{\phi'}
=
\sum_{\phi} \E I_\phi
 \sum_{\phi':H_\phi\cap H_{\phi'}=\emptyset} \E\bigpar{I_{\phi'}\mid I_\phi=1}
\notag\\&
=
\sum_{\phi} \E I_\phi
\frac{1}{((N-2e(H)-1))_{e(H')}}
\prod_{k\ge0}(n_k-h_k)_{h'_k} k!^{h_k'}
\notag\\&
=
\E \bZZ_H
\frac{1}{((N-2e(H)-1))_{e(H')}}
\prod_{k\ge0}(n_k-h_k)_{h'_k} k!^{h_k'}
\notag\\&
=
\E \bZZ_H \E \bZZ_{H'}
\frac{((N-1))_{e(H')}}{((N-2e(H)-1))_{e(H')}}
\prod_{k\ge0}\frac{(n_k-h_k)_{h'_k}}{(n_k)_{h'_k}}.
\label{moc}
\end{align}

If $H_\phi$ and $H_{\phi'}$ are not disjoint, then both can occur as
components only if they coincide as unlabelled graphs.
Hence, 
if $H\neq H'$, then $\E\bigpar{\bZZ_H\bZZ_{H'}}$ is given by \eqref{moc},
while if $H=H'$, then there is an additional term $\aut(H)\E\bZZ_H$.
We switch to counting unlabelled copies, using \eqref{uZ} as usual, and
obtain
\begin{align}
\E \bigpar{\ubZZ_H \ubZZ_{H'}}
&
=
\gd_{H,H'}\E \ubZZ_H 
\notag\\
&\hskip-2em{}
+
\E \ubZZ_H \E \ubZZ_{H'}
\frac{((N-1))_{e(H')}}{((N-2e(H)-1))_{e(H')}}
\prod_{k\ge0}\frac{(n_k-h_k)_{h'_k}}{(n_k)_{h'_k}}.
\label{mod}
\end{align}

Consider first the case when $H$ and $H'$ are unicyclic. 
If $p_k>0$ for every $k$ such that $h_kh_k'>0$, then all fractions in
\eqref{mod} tend to 
1 as \ntoo, and thus \eqref{mod} and \eqref{lmc} yield
\begin{align}
\E \bigpar{\ubZZ_H \ubZZ_{H'}}
&
\to
\gd_{H,H'}\gl_H 
+
\gl_H\gl_{H'}.
\label{moe}
\end{align}
If $p_k=0$ for some $k$ with $h_kh'_k>0$, then $\E\ubZZ_H\to\gl_H=0$ by
\eqref{lmc} and the definition \eqref{glh} of $\gl_H$ (or directly by
\eqref{mo}), and \eqref{mod} implies that  \eqref{moe} still holds.
If $H\neq H'$, 
\eqref{moe} yields \eqref{lmc2} with $k=2$ and $r_1=r_2=1$; if $H=H'$,
\eqref{moe} yields $\E \bigpar{\ubZZ_H}_2\to\gl_H^2$,
another instance of \eqref{lmc2} ($k=1$ and $r_1=2$).

This shows \eqref{lmc2} when the degree $\sum_ir_i=2$, The general case is
proved in the same way. The factorial moments in \eqref{lmc2} means that we
only count copies that are distinct, and therefore disjoint, and we may 
condition on one indicator $I_\phi$ as in \eqref{moc} and use induction;
this is a standard argument and we omit the details.

Finally, consider the case of two trees $H$ and $H'$. Then \eqref{mod} yields
\begin{align}
&\Cov \bigpar{\ubZZ_H, \ubZZ_{H'}}
=
\gd_{H,H'}\E \ubZZ_H 
\notag\\
&\hskip2em{}
+
\E \ubZZ_H \E \ubZZ_{H'}
\Bigpar{\frac{((N-1))_{e(H')}}{((N-2e(H)-1))_{e(H')}}
\prod_{k\ge0}\frac{(n_k-h_k)_{h'_k}}{(n_k)_{h'_k}}
-1}.
\label{mof}
\end{align}
If $p_k>0$ whenever $h_kh_k'>0$, then 
\begin{align}
\hskip2em&\hskip-2em
\frac{((N-1))_{e(H')}}{((N-2e(H)-1))_{e(H')}}
\prod_{k\ge0}\frac{(n_k-h_k)_{h'_k}}{(n_k)_{h'_k}}
\notag\\&
=\Bigpar{1+\frac{2e(H)e(H')}{N}+O\bigpar{N\qww}}
\prod_{k\ge0}\Bigpar{1-\frac{h_kh'_k}{n_k}+O\bigpar{n_k\qww}}
\notag\\&
= 1 +\frac{2e(H)e(H')}{\mu n} -
\sum_{k\ge0}\frac{h_kh'_k}{p_k n}+o\bigpar{n\qw}.
\label{mog}
\end{align}
Hence, \eqref{mof} and \eqref{lmc} then yield
\begin{align}
\Cov \bigpar{\ubZZ_H, \ubZZ_{H'}}
=
\gd_{H,H'} n \gl_H
+
n \gl_H\gl_{H'}
\Bigpar{\frac{2e(H)e(H')}{\mu } -\sum_{k\ge0}\frac{h_kh'_k}{p_k}}
+o\xpar{n}.
\label{moh}
\end{align}
If $p_k=0$ so $n_k=o(n)$ for some $k$ with $h_kh_k'>0$, then \eqref{bko}
yields
$\E\ubZZ_H,\E\ubZZ_{H'}=O(n_k)$, 
and it is easily seen from \eqref{mof} and an expansion as in \eqref{mog}
that $\Cov\bigpar{\ubZZ_H,\ubZZ_{H'}}=o(n)$. Hence, \eqref{moh} 
holds in this case too.

We have shown \eqref{moh} in general; this is the same as \eqref{lmb} with 
the definition \eqref{gshh}, which completes the proof of \refL{LM}.
\end{proof}

\begin{remark}\label{Ru}
We used labelled copies in the proof for convenience and transferred the
results to unlabelled copies by \eqref{uZ}.
Alternatively, and essentially equivalently,
we may  define $\phi_1\equiv\phi_2$ if $\phi_1,\phi_2\in\Phi(H)$ and
$H_{\phi_1}=H_{\phi_2}$ as unlabelled graphs, and note that then
$I_{\phi_1}=I_{\phi_2}$.  Hence $\Phi(H)$ splits into
equivalence classes of $\aut(H)$ elements each. Define
$\Phiu(H)$ as a subset of $\Phi(H)$ consisting of one element from each
equivalence class. Then, \cf{} \eqref{moa} and \eqref{uZ}, 
$|\Phiu(H)|=\aut(H)\qw|\Phi(H)|$ and
\begin{align}\label{moab}
  \ubZZ_H=\sum_{\phi\in\Phiu(H)} I_\phi,
\end{align}
which can be used instead of \eqref{moa}, yielding the same results.
\end{remark}

\begin{proof}[Proof of \refL{LM3}]
This is similar to the proof of the special case \eqref{lmc2} in \refL{LM}.
  
The special case $k=0$ is shown in the proof of \eqref{pg} in 
\cite[Section 7 including Remark 6]{SJ195}.
In general, we take copies $H_{\phi_{ij}}$  of $H_i$ for $i=1,\dots, k$ and
$j=1,\dots,r_i$, and condition on $I_{\phi_{ij}}=1$ for all such $i$ and
$j$.
Then 
\begin{align}
    \E \Bigpar{\bigpar{\uZZ_{\sC_1}}_{s_1}\bigpar{\uZZ_{\sC_2}}_{s_2}\Bigm|
I_{\phi_{ij}}=1 \forall i,j}
\to \glscx1^{s_1}\glscx2^{s_2}
\end{align}
by the case $k=0$ just discussed applied to 
$\ggndd\setminus\bigcup_{i,j}H_{\phi_{i,j}}$, 
and \eqref{lm3} follows by the argument in \eqref{moc}.
(It is here convenient to count unlabelled copies as in \refR{Ru}.)
\end{proof}

\begin{proof}[Proof of \refL{LM1.5}]
  Let $V$ be a uniformly random vertex in $\ggndd$. Then,
  \begin{align}
    \P\bigpar{\cC(V)\cong H\mid\ggndd}
= |H|\ubZZ_H/n,
  \end{align}
since each component isomorphic to $H$ contains $|H|$ vertices that are
possible choices of $V$.
Hence, using \eqref{lma},
  \begin{align}\label{ria}
    \P\bigpar{\cC(V)\cong H}
= |H|\E \ubZZ_H/n \to |H|\gl_H.
  \end{align}
Furthermore, by the coupling of the exploration process and the branching
process $\cT$ discussed in \refSS{SGW},
  \begin{align}\label{rib}
    \P\bigpar{\cC(V)\cong H}
=     \P\bigpar{\cT\cong H}+o(1)
=p_H+o(1).
  \end{align}
By \eqref{ria} and \eqref{rib}, $|H|\gl_H=p_H$.
\end{proof}

\begin{proof}[Proof of \refL{LM2}]
Suppose that this fails, and that $H_1,\dots,H_k$ are distinct 
trees with $v(H_i)>1$,
$\gl_{H_i}>0$ and $\sum_{i,j}a_ia_j\gs_{H_i,H_j}=0$ for some real numbers
$a_i\neq0$. Hence, by \eqref{cov},
\begin{align}
  \Var\Bigpar{\sumik a_i \ubZZ_{H_k}} 
= n\sum_{i,j=1}^ka_ia_j\gs_{H_i,H_j} + o(n)
= o(n).
\end{align}
Hence, for any tree $H$, by the \CSineq{} and \eqref{cov},
\begin{align}
  \Cov\Bigpar{\sumik a_i \ubZZ_{H_k},\ubZZ_H} 
\le \Var\Bigpar{\sumik a_i \ubZZ_{H_k}}\qq \Var\bigpar{\ubZZ_H}\qq 
= o(n),
\end{align}
and thus by \eqref{cov} again,
\begin{align}\label{byz}
  \sumik a_i \gs_{H_i,H}=0.
\end{align}

For two trees $T_1$ and $T_2$, define
\begin{align}\label{byaj}
\innprod{T_1,T_2}
:=
\frac{2e(T_1)e(T_2)}{\mu } -\sum_{k\ge0}\frac{n_k(T_1)n_k(T_2)}{p_k},
\end{align}
so that \eqref{gshh} can be written
\begin{align}\label{byy}
\gs_{T_1,T_2}=
\gd_{T_1,T_2}  \gl_{T_1}
+
\gl_{T_1}\gl_{T_2}\innprod{T_1,T_2}.
\end{align}
Thus, \eqref{byz} yields, for every tree $H$,
\begin{align}\label{byx}
  \sumik a_i\gl_{H_i}\gd_{H,H_i} + \gl_H\sumik a_i\gl_{H_i}\innprod{H_i,H}=0.
\end{align}
By assumption \ref{Ap1}, there exists an $r>1$ such that $p_r>0$.
Given any tree $T$ with $v(T)>1$, we may replace a leaf by a vertex of
degree $r$, joined 
to $r-1$ new leaves. Denote the result by $T\xx1$; this is a tree with
$n_k(T\xx1)=n_k(T)+(r-2)\gd_{k1}+\gd_{kr}$ and $e\bigpar{T\xx1}=e(T)+r-1$.
Hence, for any other tree $T'$,
\begin{align}\label{bya}
  \innprod{T\xx1,T'}
=\innprod{T,T'}+\frac{2(r-1)e(T')}\mu-\frac{(r-2)n_1(T')}{p_1}
 -\frac{n_r(T')}{p_r}. 
\end{align}
Repeat this procedure and obtain a sequence of trees $T\xx j$, $j\ge0$, with
$T\xx0:=T$ 
(replacing an arbitrary leaf each time). If $j$ is large enough, then
$v(T\xx j)>v(H_i)$ for $i=1,\dots,k$, and thus $T\xx j\neq H_i$ and
\eqref{byx} yields
\begin{align}\label{byb}
  \gl_{T\xx j}\sumik a_i\gl_{H_i}\innprod{H_i,T\xx j}=0,
\qquad j \text{ large}.
\end{align}
If $\gl_T>0$, then also $\gl_{T\xx j}>0$ for every $j\ge1$, and thus
\eqref{byb} yields $\sumik a_i\gl_{H_i}\innprod{H_i,T\xx j}=0$ for large $j$.
Furthermore, it follows from \eqref{bya} that this sum is a linear function
of $j$, and thus it vanishes for all $j$, i.e., 
\begin{align}\label{byj}
\sumik a_i\gl_{H_i}\innprod{H_i,T\xx j}=0,
\qquad j\ge0.
\end{align}
In particular we can take $j=0$ in \eqref{byj}, and see that
$\sumik a_i\gl_{H_i}\innprod{H_i,T}=0$ for every tree $T$ with $v(T)>0$ and
$\gl_T>0$.
Hence, the second term in \eqref{byx} vanishes for every tree $H$ with
$v(H)>1$, and thus 
\eqref{byx} and \eqref{byj} imply
\begin{align}\label{byf}
  \sumik a_i\gl_{H_i}\gd_{H,H_i} =0
\end{align}
for every such tree $H$.
However, taking $H=H_1$, this yields $a_1\gl_{H_1}=0$, a contradiction. 
\end{proof}

\section{Proof of \refT{TGpsi}}\label{Strunc}

We say that a graph functional $\psi$ has 
\emph{finite support} if $\psi(H)\neq0$ for only finitely many unlabelled $H$,
or equivalently, if there exists $K<\infty$ such that 
\begin{align}\label{supportK}
  \psi(H)=0\qquad \text{if } e(H)\ge K.
\end{align}

\begin{lemma}\label{LG1}
Assume \refAA.
  Let $\psi$ be a graph functional with finite support and define $\Psi$ by
  \eqref{Psi}.
Then
\begin{align}\label{lg1}
  \E \Psi\bigpar{\ggndd} = n\E \psi(\cT)+o(n).
\end{align}
\end{lemma}

\begin{proof}
  Let $V$ be a uniformly random vertex in $\ggndd$,
  and couple the exploration
  process of $\cC(V)$ with $\cT$ as in \refSS{SGW}.
  Let $K$ be as in \eqref{supportK}.
If the first $K$ generations of the two processes are equal, then either
$\cC(V)=\cT$ (as unlabelled rooted graphs), or both have at least $K$ edges,
and in both cases $\psi(\cC(V))=\psi(\cT)$.
Hence, noting that \eqref{supportK} also implies that $\psi$ is bounded,
\begin{align}\label{lx0}
  \E \psi\bigpar{\cC(V)} = \E \psi(\cT) + o(1).
\end{align}
Furthermore, conditioning on $\ggndd$,
\begin{align}\label{lx1}
  \E \bigpar{\psi\xpar{\cC(V)}\mid\ggndd}
  =\frac{1}{n}\sumin \psi\bigpar{\cC(v_i)}
= \frac{1}{n}\Psi\bigpar{\ggndd}.
\end{align}
Thus, taking the expectation,
\begin{align}\label{lxx}
  \E \psi\xpar{\cC(V)}
= \frac{1}{n}\E\Psi\bigpar{\ggndd}.
\end{align}
The result
\eqref{lg1} follows by  combining \eqref{lxx}
with \eqref{lx0}.
\end{proof}

\begin{lemma}\label{LG2}
Assume \refAA. Let $\psi$ be a graph functional with finite support, let $K$
be as in \eqref{supportK} and let $M:=\sup_H|\psi(H)|<\infty$.
Then, for $n$ so large that $N\ge \mu n/2$ and $N\ge 4K$,   
with $c:=4+16/\mu$, 
\begin{align}\label{lg2}
  \Var\bigpar{\Psi\bigpar{\ggndd}} \le c M K^2 \E \bigabs{\psi\xpar{\cC(V)}} n.
\end{align}
\end{lemma}

\begin{proof}
Let $V_1$ and $V_2$ be independent uniformly random vertices in $\ggndd$.
Then
\begin{align}\label{lx2}
  \E\bigpar{\psi(\cC(V_1))\psi(\cC(V_2))\mid \ggndd}
&=\frac{1}{n^2}\sum_{v_1,v_2\in\ggndd} \psi(\cC(v_1))\psi(\cC(v_2)) 
\notag\\&
=\frac{1}{n^2}\Psi\bigpar{\ggndd}^2.
\end{align}

We reveal the edges in $\ggndd$ in a special order.
(Cf.\ the related argument in \cite[Section 4.2]{BallNeal}.)
First, let $\Pi_1$ and $\Pi_2$ denote independent exploration processes of 
$\cC(V_1)$ and $\cC(V_2)$, starting at $V_1$ and $V_2$ as above.
(These may thus conflict. Think of them as exploring different copies of
$\ggndd$.) Let $Y_1:=\psi(\cC(V_1))$ and $Y_2:=\psi(\cC(V_2))$ be given by 
$\Pi_1$ and $\Pi_2$, respectively, 
and note that $Y_1$ and $Y_2$ are independent.
Next, start from scratch and reveal first the edges of $\cC(V_1)$ according
to the process $\Pi_1$, but stop when $K$ edges have been found, or when
$\cC(V_1)$ is exhausted (if it has less than $K$ edges), where again $K$ is
as in \eqref{supportK}; let $\cC_K(V_1)$ denote the explored part of $\cC(V_1)$.
Then reveal edges according to $\Pi_2$ (starting from $V_2$) as long this
does not involve any half-edge already paired. When the first conflict
occurs, abandon $\Pi_2$ and  pair the remaining half-edges uniformly at
random in any order. This yields a copy of $\ggndd$, and we define
$X_j:=\psi(\cC(Y_j))$, $j=1,2$,
for it.

By the construction, and \eqref{supportK}, $X_1=Y_1$.
Furthermore, $X_2\neq Y_2$ only if a conflict has occurred when revealing 
one of the first $K$ edges according to $\Pi_2$. This may happen either
because $V_2$ belongs to $\cC_K(V_1)$, or
because one of the half-edges already found during the exploration of the
first $K$ edges of
$\cC(V_2)$ is paired by $\Pi_2$
with one of the half-edges in $\cC_K(V_1)$.
Condition on  $\cC_K(V_1)$.
Since we reveal only at most $K$ edges and thus $K+1$ vertices in
$\cC_K(V_1)$, the first possibility has probability $\le (K+1)/n$, and the
second possibility has probability $\le 2K/(N-2K)$ for each of at most
$K$ pairings; hence, we obtain the union bound, using $N-2K\ge N/2\ge \mu n/4$,
\begin{align}\label{lx3}
  \P\bigpar{X_2\neq Y_2\mid \cC_K(V_1)} 
\le \frac{K+1}{n} +  \frac{2K^2}{N-2K}
\le \frac{2K^2}{n}\bigpar{1+4\mu\qw}.
\end{align}

The construction shows that $X_1$ and $X_2$ have the correct joint distribution,
while $Y_1$ and $Y_2$ have same individual distribution as these, but are
independent of each other. 
Thus, by \eqref{lx2} and \eqref{lx1},
\begin{align}
  \E \bigpar{X_1X_2}
&=
  \E \bigpar{\psi(\cC(V_1))\psi(\cC(V_2))}
= n^{-2} \E \bigpar{\Psi\bigpar{\ggndd}^2},
\\
  \E \bigpar{Y_1Y_2}
&=
\bigpar{\E Y_1}^2
=  \bigpar{\E \psi(\cC(V_1))}^2
= n^{-2} \bigpar{\E \Psi\bigpar{\ggndd}}^2.
\end{align}
Consequently, recalling $X_1=Y_1$,
\begin{align}\label{lx4}
n\qww \Var\bigpar{\Psi(\ggndd)}
= \E (X_1X_2)-\E(Y_1Y_2)
= \E\bigpar{X_1(X_2-Y_2)}.
\end{align}
Furthermore, by \eqref{lx3}, since $X_1$ is determined by $\cC_K(V_1)$,
\begin{align}
  \E\bigpar{|X_2-Y_2|\mid X_1}
\le 2M   \P\bigpar{X_2\neq Y_2\mid X_1} \le \frac{4MK^2}{n}\bigpar{1+4\mu\qw}.
\end{align}
Hence, \eqref{lx4} yields
\begin{align}
  n\qww\Var\bigpar{\Psi\bigpar{\ggndd}}
\le \E \bigpar{|X_1||X_2-Y_2|}
\le \frac{cMK^2}{n}\E|X_1|,
\end{align}
which yields \eqref{lg2}.
\end{proof}

\begin{lemma}  \label{LJ1}
  Assume \refAA.
  For every $\eps>0$, there exists $\eps_1>0$ such that if $\cJ\subset[n]$
  is any subset with $|\cJ|\le\eps_1n$, then $\sum_{i\in\cJ}d_i<\eps n$.
  In particular, every subgraph $H$ of $\ggndd$ with $|H|\le \eps_1n$ has
  $e(H)<\eps n$ (deterministically).
\end{lemma}
\begin{proof}
  We have $D_n:=d_I$, where $I$ is a uniformly random index in $[n]$.
  The uniform integrability of $D_n$ (see \ref{ADmu}) means
  (see \eg{} \cite[Theorem 5.4.1]{Gut})
  that there
  exists $\eps_1$ such that for any event $\cE$ with $\P(\cE)\le\eps_1$, we
  have $\E\bigpar{D_n;\,\cE}<\eps$.
  Let $\cE:=\set{I\in\cJ}$. Then $\P(\cE)=|\cJ|/n$ and
  $\sum_{i\in\cJ}d_i=n\E\bigpar{D_n;\,\cE}$.
\end{proof}

We consider now only the supercritical case
\ref{Asuper}. 

\begin{lemma}\label{LG3}
  Assume \refAA{} and \ref{Asuper}.
Then there exist $c>0$, $\eps>0$ and $C<\infty$ such that, for $\ggndd$,
\begin{align}\label{lg3}
  \P\bigpar{e\bigpar{\cC(V)}=\ell} &\le C e^{-c\ell},
\qquad
0\le \ell \le \eps n,
\\\label{lg3k}
  \P\bigpar{\abs{\cC(V)}=k} &\le C e^{-ck},
\qquad
1\le k \le \eps n.
\end{align}
\end{lemma}

\begin{proof}
Consider the exploration process, starting at a random vertex $V$ and
restarting at a new random vertex when a component is completely explored.
Let $Q_j$ be the number of unpaired half-edges in the explored part when $j$
pairings have been made. In particular, $Q_0=d(V)\eqd D_n$.

 By \ref{Asuper}, $\E D^2>2\mu$, and thus there exists $K<\infty$ such that 
\begin{align}
s_K:=  \sum_{k=1}^K k^2 p_k>2\mu.
\end{align}
Let $\eps=(s_K-2\mu)/(10K^3)>0$.
Consider until further notice 
only $n$ that are so large that $n_k/n>p_k-\eps$ for
$k=1,\dots,K$, and also $N/n < \mu+\eps$, see  \ref{AD} and \eqref{N2}.

Let $k\le K$ be such that $p_k\ge 3\eps$.
During the first $\eps n$ steps of the exploration process, there is always
at least $n_k-\eps n > (p_k -2\eps) n$ unused vertices of degree $k$, and
thus the probability that the next pairing is with a half-edge at an unused 
vertex
of degree $k$ is at least, noting that $p_k\le \mu$,
\begin{align}\label{lynx}
  \frac{k (n_k-\eps n)}{N} >\frac{k(p_k -2\eps)n}{(\mu+\eps)n}
>\frac{k(p_k-2\eps)(1-\eps/\mu)}{\mu}
>\frac{k(p_k-3\eps)}{\mu}.
\end{align}

Let $\xi_1,\xi_2,\dots$ be independent copies of a random
variables $\xi$ with
the distribution
\begin{align}\label{leo}
  \P(\xi=k)=
  \begin{cases}
    k (p_k-3\eps)_+/\mu, & 1\le k\le K,
\\
1-\sum_{j=1}^K\P(\xi=j), &k=0,
  \end{cases}
\end{align}
and define $S_m:=\sum_{j=1}^m (\xi_j-2)$.

By \eqref{lynx}, we can couple the exploration process with the variables
$\xi_i$ such 
that for every $j\le\eps n$,
$Q_{j+1}-Q_j \ge \xi_j-2 = S_{j+1}-S_j$, and thus, by induction,
$Q_j\ge S_j$.
 
Suppose now that the component $\cC(V)$ has exactly $\ell$ edges. Then
$Q_\ell=0$, and hence $S_\ell\le 0$.
However,
\eqref{leo} implies that
\begin{align}
  \mu \E\xi 
=\sum_{k=1}^K k^2(p_k-3\eps)_+
\ge \sum_{k=1}^K k^2 p_k -3 K^3\eps
>2\mu
\end{align}
and thus $\E\xi>2$.
Furthermore, $\xi$ is bounded, and thus 
$\psiq(t):=\E e^{t(\xi-2)}<\infty$ for every real
$t$. 
Hence, $\psiq'(0)=\E(\xi-2)>0$, and thus there exist $t_0<0$ such that
$\psiq(t_0)<1$. Consequently, using a Chernoff bound,
\begin{align}
\P\bigpar{e(\cC(V))=\ell}\le
\P(Q_\ell=0)\le 
  \P(S_\ell\le 0) \le \E e^{t_0 S_{\ell}} = \psiq(t_0)^\ell, 
\qquad \ell\le\eps n.
\end{align}
This proves \eqref{lg3}, with $c:=-\log\psiq(t_0)>0$, when $n$ is large
enough.
The result \eqref{lg3} extends to all $n$, since it is trivial for small $n$
if $C$ is large enough.

We turn to \eqref{lg3k}.
Let $\eps_1$ be as in \refL{LJ1}.
Then, 
any component $\cC$ with
$|\cC|=k\le \eps_1 n$ has at most $\eps n$ edges, and at least $k-1$.
Consequently, by \eqref{lg3}, if $k\le \eps_1 n$,
\begin{align}
  \P\bigpar{|\cC(V)|=k}
\le
\sum_{\ell=k-1}^{\eps n} \P\bigpar{e(\cC(V))=\ell}
\le
\sum_{\ell=k-1}^{\eps n} C e^{-c\ell}
\le C'e^{-c k}.
\end{align}
Hence \eqref{lg3k} too holds, if we redefine $C$ and $\eps$.
\end{proof}

We have a similar estimate for the supercritical branching process $\cT$.

\begin{lemma}
  \label{LG4}
  Assume \refAAsuper.
  Then, for the branching process $\cT$ in \refSS{SGW},
  \begin{align}\label{pippi}
  \P\bigpar{|\cT|=k}\le Ce^{-ck},\qquad 1\le k<\infty.
\end{align}
\end{lemma}
\begin{proof}
This follows by standard branching process theory.

Alternatively, 
  by the coupling in \refSS{SGW}, or by \eqref{lx0} with
$\psi(T):=\indic{|T|=k}$, we have for every fixed $k\ge1$,
$\P\bigpar{|\cC(V)|=k}\to\P\bigpar{|\cT|=k}$, and thus \eqref{pippi} follows
from \eqref{lg3k}.
\end{proof}

\begin{proof}[Proof of \refT{TGpsi}]
  The proofs of the two parts are essentially identical, and we give the
  details only for \ref{TGpsi*}.

  \pfitemref{TGpsi*}
Write $\GG:=\ggndd$.

We truncate. Define, for $\ell,L\ge1$,
\begin{align}
  \psi_\ell(H)&:=\psi(H)\cdot\indic{e(H)=\ell},
             \\
  \psi_{\le L}(H)&:=\psi(H)\cdot\indic{e(H)\le L},
\end{align}
and consider the corresponding $\Psi_\ell$ and $\Psi_{\le L}$.

Let $\eps>0$ be a fixed small number, chosen so small that \refL{LG3} holds,
as well as the following argument.

If $\Psix(\GG)\neq\Psi_{\lepsn}(\GG)$, then either
$e(\cC_1)\le \eps n$, and thus $|\cC_1|\le \eps n+1$,
or $e(\cC_2)>\eps n$ and thus $|\cC_2|>\eps_1 n$ by \refL{LJ1}.
If $\eps$ is small enough, then both events have probability
$O\bigpar{e^{-cn}}$ by \cite[Theorem 2]{BollobasRiordan-old}.
Since \eqref{psibound} implies
$\Psix(\GG), \Psi_{\lepsn}(\GG)=O\bigpar{n^{m+1}}$,
we thus have
\begin{align}\label{aja}
  \E\bigpar{\Psix(\GG)-\Psi_{\lepsn}(\GG)}^2
  =O\bigpar{n^{2m+2}e^{-cn}}
  = O\bigpar{e^{-cn}}.
\end{align}
Hence, it suffices to prove the results for $\Psi_{\lepsn}(\GG)$.

For every $\ell\lepsn$, \eqref{psibound} and \refL{LG3} yield
\begin{align}
  \E |\psi_{\ell}\bigpar{\cC(V)}|
  \le C\ell^m\P\bigpar{e(\cC(V))=\ell}
  \le C\ell^me^{-c\ell},
\end{align}
and thus by \eqref{lxx} and \refL{LG2} (with $M\le C\ell^m$)
\begin{align}
  \E|\Psi_\ell(\GG)|&
  =n \E |\psi_\ell(\cC(V))|
      \le Ce^{-c\ell}n,\label{ajaj}
  \\
  \Var\Psi_\ell(\GG)& \le C \ell^m\ell^2\E|\psi_\ell(\cC(V))|n
                      \le C \ell^{2m+2}e^{-c\ell}n
                      \le C e^{-c\ell}n.
                      \label{ajvar}
\end{align}

Let, for $L\ge1$,
\begin{align}
  X_n&:=n\qqw\bigpar{\Psix(\GG)-\E\Psix(\GG)},\label{baja}
       \\
  X_{n;L}&:=n\qqw\bigpar{\Psi_{\le L}(\GG)-\E\Psi_{\le L}(\GG)}.\label{caja}
\end{align}
Then, assuming $n\ge L/\eps$,
\begin{align}\label{vaja}
  X_n-X_{n;L} = X_n-X_{n;n\eps} +
  \sumlepsn n\qqw\bigpar{\Psi_{\ell}(\GG)-\E\Psi_{\ell}(\GG)},
\end{align}
and thus, using \eqref{aja}, \eqref{ajvar} and Minkowski's inequality,
\begin{align}\label{daja}
  \bigpar{\Var\xpar{X_n-X_{n;L}}}\qq
&  \le
                                       \bigpar{\Var\xpar{X_n-X_{n;\eps n}}}\qq
  +\sumlepsn n\qqw\bigpar{\Var \Psi_{\ell}(\GG)}\qq,
  \notag\\&
  \le C e^{-cn}  +\sumlepsn C e^{-c\ell}
  \le C e^{-c L}.
\end{align}

Furthermore, for every fixed $L\ge1$,
$\psi_{\le L}$ is a functional with finite support, and \eqref{Psi} yields a
finite linear combination
\begin{align}\label{maja}
  \Psi_{\le L}(\GG)=\sum_{e(H)\le L}|H|\psi(H)\ubZZ_H.
\end{align}
Terms where $H$ has a cycle have variance $O(1)$, by \refT{T1} or directly
by \refL{LL2}. Hence, they can be ignored, and \refT{T1}
(with \refR{RAp1})
yields,
by \ref{T1t} and joint convergence for different trees,
\begin{align}\label{eaja}
  X_{n;L}:=\frac{ \Psi_{\le L}(\GG)-\E  \Psi_{\le L}(\GG)}{\sqrt n}
\dto N\bigpar{0,\gss_L},
\end{align}
where
\begin{align}\label{gsspsiL}
  \gss_L&:=\sum_{|T_1|,|T_2|\le L}|T_1||T_2|\psi(T_1)\psi(T_2)\gs_{T_1,T_2},
\end{align}
summing over pairs of trees of order less than $L$.

We will verify below that the sums in \eqref{gsspsi} converge absolutely.
Thus, as \Ltoo,  $\gss_L\to\gss_\psi$ defined in \eqref{gsspsi}.
The estimate \eqref{daja}, which is uniform in $n\ge L/\eps$, implies
\begin{align}
  \lim_{\Ltoo}\limsup_{\ntoo} \E (X_n-X_{n;L})^2
  =
  \lim_{\Ltoo}\limsup_{\ntoo} \Var (X_n-X_{n;L})
  =0.
\end{align}
This together with the limit \eqref{eaja} for each fixed $L$ and
$\gss_L\to\gss_\psi$ imply,
see \eg{} \cite[Theorem 4.2]{Billingsley} or \cite[Theorem 4.28]{Kallenberg},
$X_n\dto N\bigpar{0,\gss_\psi}$, which is \eqref{tgpsi}.

Furthermore, \refT{T1} shows also that $\Var X_{n;L}\to\gss_L$ for every fixed
$L$.
Hence,  Minkowski's inequality and \eqref{daja} imply, for every fixed $L$,
\begin{align}
  \bigabs{(\Var X_n)\qq-\gs_L}
&  \le
  \bigabs{(\Var X_n)\qq-(\Var X_{n;L})\qq}
  +
  \bigabs{(\Var X_{n,L})\qq-\gs_L}
  \notag\\&
  \le
  \bigpar{\Var (X_n-X_{n;L})}\qq
  + o(1)
  \le Ce^{-cL}+o(1),
\end{align}
and thus,
\begin{align}
  \bigabs{(\Var X_n)\qq-\gs_\psi}
&  \le
    \bigabs{(\Var X_n)\qq-\gs_L}
  +  \abs{\gs_L-\gs_\psi}
  \notag\\&
  \le Ce^{-cL}+ \abs{\gs_L-\gs_\psi}
  + o(1).\label{faja}
\end{align}
Take $\limsup_\ntoo$ in \eqref{faja} and then let $L\to\infty$.
This yields
\begin{align}\label{qaja}
\limsup_\ntoo   \bigabs{(\Var X_n)\qq-\gs_\psi}=0,
\end{align}
and thus
$\Var X_n\to\gss_\psi$, which is equivalent to \eqref{tgpsi2}.

Similarly, if $n\ge L/\eps$, then \eqref{aja} and \eqref{ajaj} yield
\begin{align}\label{gaja}
  \E|\Psix(\GG)-\Psi_{\le L}(\GG)|
 & \le
  \E|\Psix(\GG)-\Psi_{\lepsn}(\GG)|+
       \sumlepsn \E|\Psi_\ell(\GG)|
       \notag\\&
  \le C e^{-cn} + C\sumlepsn Ce^{-c\ell}n
  \le Ce^{-cL}n.
\end{align}
Moreover, $\E\Psi_{\le L}(\GG)/n\to\E\psi_{\le L}(\cT)$ by \refL{LG1}, and
$\E\psi_{\le L}(\cT)\to\EE\psi(\cT)$ as \Ltoo{} by \eqref{EE}, noting that
$\EE|\psi(\cT)|<\infty$ by
\eqref{psibound} and \refL{LG4}.
Hence, for any fixed $L$, \eqref{gaja} implies
\begin{align}\label{haja}
  \hskip6em&\hskip-6em
|  \E\Psix(\GG)/n-\EE \psi(\cT)|
\le
  \E|\Psix(\GG)/n-\Psi_{\le L}(\GG)/n|
  \notag\\&{}
  + |  \E\Psi_{\le L}(\GG)/n-\E \psi_{\le L}(\cT)|
  + |\E\psi_{\le L}(\cT)-\EE\psi(\cT)|
  \notag\\&\hskip-2em
  \le Ce^{-cL}+ o(1)   + |\E\psi_{\le L}(\cT)-\EE\psi(\cT)|.
\end{align}
Taking $\limsup_\ntoo$ and then letting $\Ltoo$ yields \eqref{tgpsi1},
similarly to \eqref{qaja}.

Finally, to verify absolute convergence of the sums in \eqref{gsspsi},
suppose temporarily that $\psi(T)\ge0$, so that we may interchange order in
the summations freely. Then, considering the terms in \eqref{gshh}
separately and using \eqref{gl=p}  and \eqref{EE}, 
\begin{align}
  \hskip2em&\hskip-2em
  \sum_{T_1,T_2}|T_1||T_2|\psi(T_1)\psi(T_2)\gd_{T_1,T_2}\gl_{T_1}
  =  \sum_{T}|T|^2\psi(T)^2\gl_{T}
  =  \sum_{T}p_T|T|\psi(T)^2
\notag\\&
  =\EE\bigpar{|\cT|\psi(\cT)^2},\label{pippa}
\\
  \hskip2em&\hskip-2em
  \sum_{T_1,T_2}|T_1||T_2|\psi(T_1)\psi(T_2)\gl_{T_1}\gl_{T_2}e(T_1)e(T_2)
=\Bigpar{\sum_Tp_T\psi(T)e(T)}^2
\notag\\&
=\bigpar{\EE\bigpar{\psi(\cT)e(\cT)}}^2,\label{pippb}
  \\
  \hskip2em&\hskip-2em
             \sum_{T_1,T_2}|T_1||T_2|\psi(T_1)\psi(T_2)\gl_{T_1}\gl_{T_2}
             \sum_{k\ge0}\frac{n_k(T_1)n_k(T_2)}{p_k}
             \notag\\&
  =             \sum_{k\ge0}\frac{1}{p_k}\Bigpar{\sum_Tp_T\psi(T)n_k(T)}^2
  =  \sum_{k\ge0}\frac{1}{p_k}\Bigpar{\EE\bigpar{\psi(\cT)n_k(\cT)}}^2.
  \label{pippc}
\end{align}
The expectations in \eqref{pippa} and \eqref{pippb} are finite by
\eqref{psibound} and \refL{LG4}, and so is each expectation in \eqref{pippc}.
If also the sum in \eqref{pippc} converges, then \eqref{gshh} and
\eqref{pippa}--\eqref{pippc} show the last equality in \eqref{gsspsi}.
In particular,  if \eqref{pippc} is finite, then \eqref{gsspsi} yields,
since $\gss_\psi\ge0$,
\begin{align}\label{pippe}
  \sum_{k\ge0}\frac{1}{p_k}\bigpar{\EE\xpar{n_k(\cT)\psi(\cT)}}^2
  \le\EE\bigpar{|\cT|\psi(\cT)^2}
+\frac{2}{\mu}\bigpar{\EE \xpar{e(\cT)\psi(\cT)}}^2.
\end{align}
For any $\psi\ge0$, this applies to the truncation $\psi_{\le L}$ for any $L$,
since we have
$\psi_{\le L}(T)n_k(T)=0$ for every $T$ when $k>L$, and thus the sum over
$k$ in \eqref{pippc} converges. Hence \eqref{pippe} holds for $\psi_{\le L}$, and
letting $L\to\infty$ shows (by monotone convergence) that \eqref{pippe}
holds for $\psi$ too. We have already seen that the \rhs{} in \eqref{pippe}
is finite, and thus the sum on the left converges, for any $\psi\ge0$
satisfying \eqref{psibound}.

Consequently, all sums in \eqref{pippa}--\eqref{pippc} converge when
$\psi\ge0$.
Hence, applying this to $|\psi|$, we see that for every $\psi$ satisfying
\eqref{psibound}, all sums in \eqref{pippa}--\eqref{pippc} converge
absolutely; hence the equalities in \eqref{pippa}--\eqref{pippc}
hold in general, which verifies \eqref{gsspsi}, with absolute convergence
everywhere.

\pfitemref{TGpsi0}
The main difference is that we use \refT{T2} instead of \refT{T1}.
Since we now assume \ref{AD2}, \eqref{pg} holds, and thus all 
estimates of means and variances
in the proof of \ref{TGpsi*} hold automatically for $\gndd$ too by conditioning.
\end{proof}

\begin{proof} [Proof of \refT{TG}]
  As said in \refE{Egiant}, we apply \refT{TGpsi} with  $\psi(H):=1$,
  together with \refLs{LG00} and  \ref{LGvar} below.
  \end{proof}

  \section{Non-zero variance?}\label{Snull}

The asymptotic variance $\gss_\psi$ in \refT{TGpsi} necessarily satifies
$\gss_\psi\ge0$; however, $\gss_\psi=0$ is possible.
We note first some trivial cases.
\begin{example}\label{Etrivial}
  \begin{alphenumerate}
    
  \item If \ref{Ap1} does not hold, so (in the supercritical case) $p_1=0$,
    then $\gss_\psi=0$ for every $\psi$, see \refR{RRAp1}.
  \item\label{Et1}
    If $\psi(T)=\gd_{T,\sK_1}$, then $\gss_\psi=0$.
    In fact, then $\Psi(G)=n_0(G)$ counts isolated vertices, and thus
$\Psix(\gndd)$ is deterministic.
\item\label{Etc}
  If $\psi(T)=0$ for every tree $T$ with $p_T>0$, then
$\psi(\cT)=0$ \as{} and thus $\gss_\psi=0$.
In this case, 
  $\psi(\cC_j)=0$ for
all but a few components $\cC_j$.
  \end{alphenumerate}
\end{example}

Parts \ref{Et1} and \ref{Etc} of \refE{Etrivial} show that the values of
$\psi(H)$ for $H=\sK_1$, trees $H$ with $p_H=0$, and non-trees $H$, do not
affect $\gss_\psi$.

We conjecture that the trivial cases in \refE{Etrivial}
(in combination) is the only way to get $\gss_\psi=0$ (in the supercritical
case). Formally:

\begin{conjecture}\label{Conj0}
  If \refAAAsuperp{} hold, then
  \begin{align}
  \gss_\psi=0\iff \psi(T)=0\text{ for every tree $T$ with $|T|>1$ and $p_T>0$.}  
  \end{align}
\end{conjecture}

\begin{remark}
  \label{R0}
  As said in \refS{Sgiant}, we consider for simplicity only the
  supercritical case in \refT{TGpsi}, although we expect a similar result
  (for $\Psi$) also in the subcritical case under suitable
  conditions. However, note that then there are two further trivial cases
  with $\gss_\psi=0$,
  \viz{}
  $\psi(H)=1$ and $\psi(H)=e(H)/|H|$ as in \refE{Egiant} (and linear
  combinations of them), since then $\Psi(G)=|G|$ and $e(G)$, respectively,
  which are deterministic for $\gndd$.
\end{remark}

Although we have not been able to verify \refConj{Conj0}, we can show
$\gss_\psi>0$ in many cases.

\begin{lemma}\label{LG0}
  Assume \refAAsuperp.
  Suppose that the graph functional $\psi$ satisfies 
  \eqref{psibound} and that $\gss_\psi=0$.
  Then $\psi$ has the form,
  for some real constants $a_k$ and every tree $T$ with $p_T>0$, 
\begin{align}\label{0a}
  \psi(T)=\frac{1}{|T|}\sumko a_k n_k(T),
\end{align}
where, furthermore, for every $k\ge0$, 
\begin{align}\label{0b}
  p_ka_k &=
      \EE\bigpar{n_k(\cT)\psi(\cT)}
           -\frac{kp_k}{\mu}\EE \bigpar{e(\cT)\psi(\cT)}
           \\&
=  \sumjo a_j \EE\Bigpar{\frac{n_k(\cT)n_j(\cT)}{|\cT|}
  -\frac{kp_k}{\mu}\frac{e(\cT)n_j(\cT)}{|\cT|}}.
\label{0c}
\end{align}
\end{lemma}

Note that \eqref{0c} is an eigenvalue equation for the vector $(a_k)_k$.

\begin{proof}
  Let $\PPsi$ be the linear space of all graph functionals satisfying
  \eqref{psibound}. 
  The \rhs{} of \eqref{gsspsi}  is a quadratic form in $\psi\in\PPsi$.
  Denote the corresponding symmetric
  bilinear form by $\innprod{\psi_1,\psi_2}$; thus
  $\gss_\psi=\innprod{\psi,\psi}$ and
\begin{align}\label{gsspsi2}
  \innprod{\psi_1,\psi_2}&
=\EE\bigpar{|\cT|\psi_1(\cT)\psi_2(\cT)}
 +\frac{2}{\mu}\EE \xpar{e(\cT)\psi_1(\cT)}\EE \xpar{e(\cT)\psi_2(\cT)}
  \notag\\&
  \qquad -\sum_{k\ge0}\frac{1}{p_k}\EE\xpar{n_k(\cT)\psi_1(\cT)}
    \EE\xpar{n_k(\cT)\psi_2(\cT)}.
\end{align}

Since \eqref{tgpsi2} implies $\gss_\psi\ge0$, the bilinear form
\eqref{gsspsi2}
is positive semidefinite, and thus the \CSineq{} holds for it. In particular, if
$\innprod{\psi,\psi}=\gss_\psi=0$, then $\innprod{\psi,\psi'}=0$ for every
graph functional $\psi'\in\PPsi$. Hence, taking
$\psi'(H):=\indic{H\cong T}$ for a tree $T$, 
\begin{align}\label{gsspsi3}
0
=p_T|T|\psi(T)
 +\frac{2}{\mu}\EE \xpar{e(\cT)\psi(\cT)}p_Te(T)
  -\sum_{k\ge0}\frac{1}{p_k}\EE\xpar{n_k(\cT)\psi(\cT)}
    p_Tn_k(T).
\end{align}
 Thus, for every tree $T$ such that $p_T>0$, using $2e(T)=\sum_k kn_k(T)$,
\begin{align}\label{psi4}
|T|\psi(T)
&=   \sum_{k\ge0}\frac{1}{p_k}\EE\xpar{n_k(\cT)\psi(\cT)} n_k(T)
  -\frac{2}{\mu}\EE \xpar{e(\cT)\psi(\cT)}e(T)
   \notag\\
  &=   \sum_{k\ge0} \Bigpar{\frac{1}{p_k}\EE\xpar{n_k(\cT)\psi(\cT)}
    -\frac{k}{\mu}\EE \xpar{e(\cT)\psi(\cT)}}
    n_k(T),
\end{align}
which yields \eqref{0a}--\eqref{0b}; then \eqref{0c} follows by
    substituting \eqref{0a} in \eqref{0b}.
\end{proof}

\begin{lemma}\label{LG00}
  Assume \refAAsuperp.
Let\/ $\psi(H):=1$. Then $\gss_\psi>0$.
\end{lemma}
\begin{proof}
  Suppose that $\gss_\psi=0$; thus \eqref{0a}--\eqref{0c} hold by \refL{LG0}.
  Furthermore,
  \eqref{0a} trivially holds with $a_k=1$, and since 
  the coefficients $a_k$ in \eqref{0a} are uniquely determined for every $k$
  with $p_k>0$, \eqref{0b} yields
  \begin{align}
    p_k=p_ka_k = \EE n_k(\cT) -\frac{kp_k}{\mu}\EE e(\cT),
    \qquad k\ge0.
  \end{align}
  Summing over $k$ yields
  \begin{align}
    1=&
        \sum_k\Bigpar{ \EE n_k(\cT) -\frac{kp_k}{\mu}\EE e(\cT)}
    =\EE|\cT|-\EE e(\cT)
    =\EE\bigpar{|\cT|-e(\cT)}
    \notag\\&
    =\EE 1=\P(|\cT|<\infty).
  \end{align}
  This is a contradiction.
\end{proof}

\begin{remark}\label{RG0}
  As said in \refE{Egiant} (and more generally in \refConj{Conj0}),
  we conjecture that $\gss_\psi>0$ also for
$\psi(H):=e(H)/|H|$, and for (non-zero) linear combinations of these two graph
functionals, but we have failed to show this in general,
and leave this as an open problem.

Frank Ball (personal communication) has noted that in the case of bounded
maximum degree treated in his paper \cite{Ball2018} (on epidemics, but with the
giant component as a special case), his formulas \cite[(5.25)--(5.27)]{Ball2018}
express the asymptotic variance as a sum of integrals of squares, making it
easy to show that the asymptotic variance is strictly positive in these
cases.
It seems likely that this formula for the asymptotic variance holds more
generally (perhaps by purely algebraic manipulations), which might lead to a
general proof, but we have not checked the details.
\end{remark}

\begin{example}
  \label{Esusc0}
  Let $\psi(H):=|H|$. \ref{Ap1} shows both $p_1>0$ and the existence
  of $r>1$ with $p_r>0$. There exist arbitrarily large trees $T$ with all
  degrees in \set{1,r}, and they have $p_T>0$.
  However, \eqref{0a} cannot hold for all such trees, since the \rhs{} is
  bounded for them by $|a_1|+|a_r|$, while $\psi(T)=|T|$ is unbounded.
  Hence \refL{LG0} shows that $\gss_\psi>0$, as claimed in \refE{Esusc}.
\end{example}

  \section{The variance of the giant}\label{Svar2}

  \begin{lemma}\label{LGvar}
    If\/ $\psi(H)=1$, then the variance $\gss_\psi$ in \eqref{gsspsi} is
    given by \eqref{tgvar}.
  \end{lemma}

  \begin{proof}      
Let $\cT_1$ be the Galton--Watson tree with offspring distribution
$\Di:=\hD-1$, with $\hD$ given by \eqref{hD}. Recall that $\cT$ has a root
with a random number $D$ of copies of $\cT_1$ attached to it.
The \pgf{} of $\Di$ 
is given by, using \eqref{hD}
and \eqref{f},
\begin{align}\label{ea}
  f_1(z):=\E z^{Y_1}=\E z^{\hD-1}=\sumk \frac{kp_k}{\mu}z^{k-1}
  =\mu\qw f'(z).
\end{align}
The probability that the supercritical Galton--Watson process $\cT_1$
is finite is the unique root $\zeta\in[0,1)$ of $f_1(\zeta)=\zeta$,
which by \eqref{ea} is equivalent to \eqref{eb}.

Let $\cT_2$ be $\cT_1$ conditioned on being finite.
Then, see \eg{} \cite[Theorem 1.12.3]{AthreyaNey},
$\cT_2$ is another Galton--Watson process,
which is subcritical and has an offspring distribution $\Dii$ with \pgf{},
using \eqref{ea},
\begin{align}\label{ec}
  f_2(z):=\E z^{\Dii} =\frac{f_1(\zeta z)}{\zeta}
  =\frac{f'(\zeta z)}{\mu \zeta}.
\end{align}
In particular,
\begin{align}\label{ed}
\E\Dii=f_2'(1)
  =\frac{f''(\zeta )}{\mu}.
\end{align}
Since $\cT_2$ is subcritical, thus $f''(\zeta)/\mu<1$.

Let $\nx_{k}(\cT_2)$ be the number of vertices in $\cT_2$ with outdegree
$k$. Then, by a standard calculation for subcritical Galton--Watson trees,
summing the expected number of such vertices in generation $j\ge0$,
and using \eqref{ec}--\eqref{ed} and  \eqref{hD},
\begin{align}\label{ef}
  \E \nx_k(\cT_2)
&  =\sum_{j=0}^\infty(\E \Dii)^j\P(\Dii=k)
  =\frac{1}{1-\E\Dii}\zeta^{k-1}\P(\Di=k)
                    \notag\\&
  =\frac{\zeta^{k-1}\P(\hD=k+1)}{1-\E\Dii}
  =\frac{(k+1)p_{k+1}\zeta^{k-1}}{\mu-f''(\zeta)}.
\end{align}
If the root of $\cT$ has degree $\ell$, then the tree is finite with
probability $\zeta^\ell$, and conditioned on this event, it has
$\ell$ branches that are copies of $\cT_2$. Hence,
using \eqref{ef} and \eqref{eb},
\begin{align}
  \EE n_k(\cT)
&  = \sum_{\ell=0}^\infty p_l\zeta^\ell\bigpar{\gd_{k\ell}+\ell\E\nx_{k-1}(\cT_2)}
  =p_k\zeta^k+
  \sum_{\ell=0}^\infty p_l\zeta^\ell\ell \frac{kp_k\zeta^{k-2}}{\mu-f''(\zeta)}
                 \notag\\&
  =p_k\zeta^k+
  f'(\zeta) \frac{kp_k\zeta^{k-1}}{\mu-f''(\zeta)}
  =p_k\zeta^k+
  \frac{kp_k\zeta^{k}\mu}{\mu-f''(\zeta)}.
\end{align}
Summing over $k$ we find
\begin{align}\label{eh}
\EE|\cT|&=\sumko  \EE n_k(\cT)
  =f(\zeta)+
  \frac{\zeta f'(\zeta)\mu}{\mu-f''(\zeta)}
  =f(\zeta)+   \frac{\mu^2\zeta^2}{\mu-f''(\zeta)},
  \\
  \EE e(\cT)&=  \EE\bigpar{|\cT|-1}
              =\EE|\cT|-\P(|\cT|<\infty)
                =\EE|\cT|-f(\zeta)
              \notag\\&
    =  \frac{\mu^2\zeta^2}{\mu-f''(\zeta)}.
\end{align}
Hence, \eqref{gsspsi} with $\psi=1$ yields
\begin{align}
  \gss&=f(\zeta)+\frac{\mu^2\zeta^2}{\mu-f''(\zeta)}
  +2 \frac{\mu^3\zeta^4}{(\mu-f''(\zeta))^2}
  -\sumko p_k\Bigpar{1+k\frac{\mu}{\mu-f''(\zeta)}}^2\zeta^{2k},
\end{align}
which yields \eqref{tgvar} by expanding the square and summing,
using \eqref{f}.
  \end{proof}

  \section{Random degrees}\label{Srandom}
  One reason for the importance of the model $\gndd$ is that for several
  models of random graphs, if we condition on the degree sequence, then we
  obtain a graph of the type $\gndd$, for the observed degree sequence $\ddn$.
This includes the \ER{} graphs $\gnp$ and $\gnm$, and  several others,
see \eg{} \cite{BrittonDeijfenML}.
Such random graphs can thus be regarded as $\gndd$ based on a random degree
sequence $\ddn$. We obtain easily results for such graphs too, by
conditioning on $\ddn$. Note that, as pointed out \eg{} by \citet{BallNeal},
the randomness in the degree sequence will in general affect the asymptotic
variance.
We illustrate this by considering in some detail the counts of small
isolated trees in the basic theorem \refT{T2}.
Similar versions of \eg{} \refTs{TG} and \ref{TGpsi} follow similarly;
see \cite{BallNeal} for the variance of the size of the giant component
in \refT{TG}.

\begin{theorem}\label{TR}
  Let $(p_k)\xoo$ be a probability distribution satisfying \ref{Ap1}.
  Suppose that $\ddn$ is random, and such that, with $\mu:=\sum_k kp_k$,
  \begin{align}
    \frac{n_k-p_kn}{\sqrt n}&\dto \xi_k,
    \qquad k=0,1,\dots, \label{tr}
    \\
  \frac{\sum_k kn_k-\mu n}{\sqrt n}&\dto \sumk k\xi_k,\label{tr1}
  \end{align}
  jointly, where $\xi_k$ are jointly normal with $\E\xi_k=0$ and
  some covariances  $\Cov(\xi_k,\xi_\ell)=\gam_{k\ell}$
  with $\sum_{k,\ell}k\ell|\gam_{k\ell}|<\infty$,
  and furthermore that
  \begin{align}\label{tr2}
    n\qw\sumk k^2 n_k \pto \sum_k p_k k^2<\infty.
  \end{align}
  Then, for any tree $H$,
  \begin{align}\label{trx}
    \frac{\ubZ_H-n\gl_H}{\sqrt n}\dto N\bigpar{0,\bgss_H},
  \end{align}
  where $\gl_H$ is as in \eqref{glh}; furthermore, joint convergence holds
  for several trees $H$, with limit $N(0,\bgS)$, where the covariance matrix
  $\bgS$ is given by
  \begin{align}
    \label{bgss}
    \bgs_{H_1,H_2}:=\gs_{H_1,H_2}+\gl_{H_1}\gl_{H_2}\sum_{k,\ell=0}^\infty
    \Bigpar{\frac{n_k(H_1)}{p_k}-\frac{k}{\mu}e(H_1)}
    \Bigpar{\frac{n_\ell(H_2)}{p_\ell}-\frac{\ell}{\mu}e(H_2)}\gam_{k\ell}.
  \end{align}
\end{theorem}

\begin{proof}
  For convenience, we use 
  the Skorohod coupling theorem \cite[Theorem 4.30]{Kallenberg},
  and may thus assume that the limits in \eqref{tr}--\eqref{tr2} hold a.s.
Then \refAAA{} hold a.s., and thus \refT{T2} applies, conditioned on the
degree sequence. Hence, conditionally,
\begin{align}\label{tr4}
  \frac{\ubZ_H-\E\bigpar{\ubZZ_H\mid\ddn}}{\sqrt n}\dto N\bigpar{0,\gss_H}.
\end{align}
Consider a tree $H$ and let $h_k:=n_k(H)$.
If $\gl_H>0$,  then, a.s., 
using \eqref{bko} and remembering that $n_k$ and $N$ now are random,
together with the a.s.\ versions of \eqref{tr} and \eqref{tr1},
\begin{align}
  \frac{\E\bigpar{\ubZZ_H\mid\ddn}}{n\gl_H}
  &=
    \frac{{1+O(n\qw)}}{n\aut(H)\gl_H}N^{-e(H)}\prod_k n_k^{h_k}k!^{h_k}
    \notag\\&
  =\bigpar{1+O(n\qw)}
   \parfrac{N}{n\mu}^{-e(H)}\prod_k \parfrac{n_k}{np_k}^{h_k}
    \notag\\&
  =
  1-\frac{e(H)}{\mu\sqrt n}\sumk k\xi_k+\sumko  \frac{h_k}{p_k\sqrt n}\xi_k
  +o\bigpar{n\qqw}
    \notag\\&
  =
  1+n\qqw\sumko\Bigpar{ \frac{h_k}{p_k}-\frac{k}{\mu}e(H)}\xi_k
  +o\bigpar{n\qqw}
  \label{tr6}
\end{align}
and thus
\begin{align}
  \frac{ \E\bigpar{\ubZZ_H\mid\ddn}-n\gl_H}{\sqrt n}
  \dto \Xi_H:=\gl_H\sumko\Bigpar{ \frac{h_k}{p_k}-\frac{k}{\mu}e(H)}\xi_k.
  \label{tr5}
\end{align}

On the other hand, if $\gl_H=0$, then $p_k=0$ for some $k$ with $h_k>0$.
Since $p_k=0$, \eqref{tr} implies $\xi_k\ge0$ \as, and thus $\xi_k=0$
a.s., so $n_k=o(n\qq)$ a.s. Hence, 
\eqref{bko} implies $\E\bigpar{\ubZZ_H\mid\ddn}=o\bigpar{n\qq}$ a.s.,
and thus \eqref{tr5} holds in this case too, with $\Xi_H=0$.

Since \eqref{tr4} holds, with the same limit, conditioned on $\ddn$, the
limits \eqref{tr4} and \eqref{tr5} hold jointly, with independent limits.
Hence, we can take their sum and obtain \eqref{trx} with
$\bgss_{H_1,H_2}:=\gss_{H_1,H_2}+\Cov\bigpar{\Xi_{H_1},\Xi_{H_2}}$, \ie,
\eqref{bgss}.
\end{proof}

\begin{example}\label{Erandomd}
  One case studied also by \citet{BallNeal} is when the degrees $d_i$ are
  \iid{} random variables that are copies of a given $D$,
  which we assume satisfies  $\E D^2<\infty$.
  (We ignore one half-edge if the sum $N$ otherwise becomes odd.)
  Then,
  by the central limit theorem and the law
  of large numbers,
  \eqref{tr}--\eqref{tr2} hold with
  \begin{align}\label{byy}
    \gam_{k\ell}=\gd_{k\ell}p_k-p_kp_l.
  \end{align}
  Hence, \refT{TR} shows that \eqref{trx} holds.
  Furthermore, the double sum in \eqref{bgss} is by \eqref{byy}
    \begin{multline}
      \sum_{k=0}^\infty p_k
      \Bigpar{\frac{n_k(H_1)}{p_k}-\frac{k}{\mu}e(H_1)}
      \Bigpar{\frac{n_k(H_2)}{p_k}-\frac{k}{\mu}e(H_2)}
      \\ -
      \sum_{k=0}^\infty p_k
      \Bigpar{\frac{n_k(H_1)}{p_k}-\frac{k}{\mu}e(H_1)}
            \sum_{\ell=0}^\infty p_\ell
    \Bigpar{\frac{n_\ell(H_2)}{p_\ell}-\frac{\ell}{\mu}e(H_2)}
.\label{k2}    \end{multline}
 By simple algebra, using $\sum_k n_k(H)=|H|$ and $\sum_k kn_k(H)=2e(H)$,
the first sum in \eqref{k2} equals
\begin{align}
        &\sum_{k=0}^\infty 
  \frac{n_k(H_1)n_k(H_2)}{p_k}
  -\frac{4e(H_1)e(H_2)}{\mu}+\frac{\E D^2}{\mu^2}e(H_1)e(H_2)
\end{align}
and the second and third are just
\begin{align}
|H_j|-e(H_j)=1.
\end{align}
Hence, \eqref{bgss} and \eqref{gshh} yield, after some interesting
cancellations,
\begin{align}
\bgs_{H_1,H_2}
=
\gd_{H_1,H_2}  \gl_{H_1}
+
\gl_{H_1}\gl_{H_2}\Bigpar{
\frac{\E D(D-2)}{\mu^2}e(H_1)e(H_2) -1}.
\label{exo}
\end{align}
\end{example}

\begin{example}\label{ER}
  Consider the \ER{} graphs
  $\gnp$ and $\gnm$, where we keep the average degree constant by
  choosing $p=\glc/n$ and $m=\glc n/2+o(n\qq)$ for
  some fixed $\glc>0$.
The number of isolated trees of a given size in \gnm{} was studied already
by \citet{ER}, but
in the range of $m$ considered here,
their result contains an error as was pointed out by \citet{Barbour}, who
proved asymptotic normality for
$\gnp$ using Stein's method. The result was extended by \citet{BarbourKR} to
counts of isolated copied of individual trees, our $\ubZ_T$;
the method yields also joint convergence.
(We are not aware of any similar result on asymptotic normality proved
for \gnm, but such results
might be in the literature.)
Although this result thus can be proved directly in a rather simple way,
at least for $\gnp$, we find it instructive to see how it follows, for both
models, from the general results in the present paper.

Asymptotic normality of $n_k$ was shown for \gnp{} by \cite{BarbourKR}, see
also \cite[Example 6.35]{JLR}.
This was extended to \gnm{}  in \cite[Theorem 4.1]{SJ190},
which implies (as a consequence of \cite[(4.2)]{SJ190}) that
\eqref{tr}--\eqref{tr2} hold for both $\gnp$ and $\gnm$ (with $p=\glc/n$ and
$m=\glc n/2+o(n\qq)$ as above), with $p_k=\glc^k e^{-\glc}/k!$, the Poisson
$\Po(\glc)$ distribution, and
asymptotic covariances
\begin{align}
  \gam_{k\ell}
  =
  \begin{cases}
    \gd_{k\ell}p_k -p_kp_\ell+ \frac{(k-\glc)(\ell-\glc)}{\glc}p_kp_\ell,
    & \gnp,
    \\[2pt]
    \gd_{k\ell}p_k -p_kp_\ell- \frac{(k-\glc)(\ell-\glc)}{\glc}p_kp_\ell,
    & \gnm.
  \end{cases}
      \label{exa}
\end{align}
Comparing with \refE{Erandomd}, we see that $\gam_{k\ell}$ has an additional
term, with different signs in the two cases.
To calculate the contribution from that term to \eqref{bgss}, we calculate,
recalling $D\sim\Po(\mu)$ so $\E D^2=\mu^2+\mu$,
\begin{align}
  \hskip2em&\hskip-2em
  \sumko \Bigpar{\frac{n_k(H)}{p_k}-\frac{k}{\mu}e(H)}(k-\glc)p_k
  =\sumko(k-\glc)n_k(H)-\sumko\frac{p_k k(k-\glc)}{\mu}e(H)
             \notag\\&
  =2e(H)-\glc |H|-e(H)
  =-(\glc-1)|H|-1.
  \label{exb}
\end{align}
Define $\chi:=+1$ for $\gnp$ and $\chi:=-1$ for $\gnm$.
Then, using \eqref{exa} in \eqref{bgss} yields, 
using \eqref{exo} and \eqref{exb} in the calculations,
  \begin{align}
&\bgs_{H_1,H_2}
=
\gd_{H_1,H_2}  \gl_{H_1}
    \notag\\&\;+
\gl_{H_1}\gl_{H_2}\Bigpar{
    \frac{\glc-1}{\glc}e(H_1)e(H_2) -1
+\frac{\chi}{\glc}\bigpar{(\glc-1)|H_1|+1}\bigpar{(\glc-1)|H_2|+1}
    }.
\label{exox}
\end{align}

Thus, \refT{TR} yields asymptotic normality of the counts of isolated trees,
in both \gnp{} and \gnm, with asymptotic covariances \eqref{exox}.  
\end{example}

\begin{remark}
  The model in \refE{Erandomd} with $D\sim\Po(\mu)$ thus gives a result
  with a covariance matrix \eqref{exo} that is half-way between the results
  for $\gnp$ and $\gnm$. As has been remarked before, the same is seen in
  the much simpler (but related) case $e(G)$ of the number of edges;
  elementary calculations yield $\Var e(G)=\Var(N/2)\sim \mu n/4$ for the
  model in \refE{Erandomd}, $\Var e(G)\sim\mu n/2$ for $\gnp$, and of course
  $\Var e(G)=0$ for \gnm.
\end{remark}

\section*{Acknowledgement}
I thank Frank Ball for helpful correspondence.

\newcommand\AAP{\emph{Adv. Appl. Probab.} }
\newcommand\JAP{\emph{J. Appl. Probab.} }
\newcommand\JAMS{\emph{J. \AMS} }
\newcommand\MAMS{\emph{Memoirs \AMS} }
\newcommand\PAMS{\emph{Proc. \AMS} }
\newcommand\TAMS{\emph{Trans. \AMS} }
\newcommand\AnnMS{\emph{Ann. Math. Statist.} }
\newcommand\AnnPr{\emph{Ann. Probab.} }
\newcommand\CPC{\emph{Combin. Probab. Comput.} }
\newcommand\JMAA{\emph{J. Math. Anal. Appl.} }
\newcommand\RSA{\emph{Random Struct. Alg.} }
\newcommand\ZW{\emph{Z. Wahrsch. Verw. Gebiete} }
\newcommand\DMTCS{\jour{Discr. Math. Theor. Comput. Sci.} }

\newcommand\AMS{Amer. Math. Soc.}
\newcommand\Springer{Springer-Verlag}
\newcommand\Wiley{Wiley}

\newcommand\vol{\textbf}
\newcommand\jour{\emph}
\newcommand\book{\emph}
\newcommand\inbook{\emph}
\def\no#1#2,{\unskip#2, no. #1,} 
\newcommand\toappear{\unskip, to appear}

\newcommand\arxiv[1]{\texttt{arXiv:#1}}
\newcommand\arXiv{\arxiv}

\end{document}